\documentclass[a4paper,10pt]{article}
\usepackage{RR}
\RRNo{6413}
\usepackage{url}
\usepackage[francais,english]{babel}
\usepackage[latin1]{inputenc}
\usepackage[T1]{fontenc}

\usepackage{latexsym}
\usepackage{amsfonts}
\usepackage{amsmath}
\usepackage{amssymb}
\usepackage{graphicx}
\usepackage{amsthm}

\newcommand{\brk}[1]{\left(#1\right)}          
\newcommand{\Brk}[1]{\left[#1\right]}          
\newcommand{\BRK}[1]{\left\{#1\right\}}        
\newcommand{\Abs}[1]{\left| #1 \right|}        
\newcommand{\Norm}[1]{\left\| #1 \right\|}     

\newcommand{\jump}[1]{[\![#1]\!]}

\newcommand{\e}{\varepsilon}
\newcommand{\D}{\mathcal{D}}
\renewcommand{\to}{\rightarrow}

\newcommand{\str}{{\boldsymbol{\tau}}}
\newcommand{\strs}{{\boldsymbol{\sigma}}}
\newcommand{\lstr}{{\boldsymbol{\psi}}}
\newcommand{\vel}{{\boldsymbol{u}}}
\newcommand{\bv}{\boldsymbol{v}}

\newcommand{\gvel}{{\grad\vel}}
\newcommand{\B}{\boldsymbol{B}}
\newcommand{\Om}{\boldsymbol{\Omega}}
\newcommand{\N}{\boldsymbol{N}}

\newcommand{\elstr}{e^{\lstr}}
\newcommand{\mlstr}{e^{-\lstr}}
\newcommand{\bphi}{\boldsymbol{\phi}}

\newcommand{\velh}{\vel_h}

\newcommand{\Omh}{\Om_h}
\newcommand{\Nh}{\N_h}
\newcommand{\ph}{p_h}
\newcommand{\strh}{\strs_h}
\newcommand{\lstrh}{\lstr_h}
\newcommand{\elstrh}{e^{\lstrh}}

\newcommand{\velhn}{\vel_h^n}
\newcommand{\velhnp}{\vel_h^{n+1}}

\newcommand{\gvelhnp}{\gvel_h^{n+1}}
\newcommand{\phn}{p_h^n}
\newcommand{\phnp}{p_h^{n+1}}
\newcommand{\strhn}{\strs_h^n}
\newcommand{\strhnp}{\strs_h^{n+1}}
\newcommand{\lstrhn}{\lstr_h^n}
\newcommand{\lstrhnp}{\lstr_h^{n+1}}
\newcommand{\elstrhn}{e^{\lstrhn}}
\newcommand{\elstrhnp}{e^{\lstrhnp}}

\newcommand{\emstrhnp}{e^{-\lstrhnp}}

\newcommand{\dt}{\Delta t}

\newcommand{\pih}[1]{\pi_h\brk{#1}}

\newcommand{\Ph}{P_h}
\newcommand{\Phrot}{P_h^{rot}}
\newcommand{\PhBDM}{P_h^{BDM}}
\newcommand{\PhRTO}{P_h^{RT_0}}

\newcommand{\I}{{\boldsymbol{I}}}

\newcommand{\x}{{\boldsymbol{x}}}

\newcommand{\bn}{\boldsymbol{n}}

\newcommand{\Wi}{{\text{Wi}}}
\renewcommand{\Re}{{\text{Re}}}

\renewcommand{\div}{\operatorname{div}}

\newcommand{\tr}{\operatorname{tr}}

\newcommand{\grad}{\boldsymbol{\nabla}}

\newcommand{\R}{\mathbb{R}}
\newcommand{\Pz}{\mathbb{P}_0}
\newcommand{\Po}{\mathbb{P}_1}
\newcommand{\Pod}{\mathbb{P}_{1,disc}}
\newcommand{\Pt}{\mathbb{P}_2}

\newcommand{\dd}{\frac{d(d+1)}{2}}

\newcommand{\pd}[2]{\frac{\partial#1}{\partial#2}}
\newcommand{\deriv}[2]{\frac{d#1}{d#2}}

\newcommand{\intd}{\int_{\D}}

\newtheorem{theorem}{Theorem}[section]
\newtheorem{lemma}{Lemma}[section]
\newtheorem{proposition}{Proposition}[section]

\newtheorem{definition}{Definition}[section]

\newtheorem{remark}{Remark}[section]

\newcommand{\comment}[1]{ }

\usepackage[cyr]{aeguill}

\RRdate{Janvier, 2008}

\RRtitle{Sch\'emas dissipatifs en \'energie libre pour le mod\`ele d'{O}ldroyd-B}

\RRetitle{Free-energy-dissipative schemes for the Oldroyd-B model}

\titlehead{Dissipative Oldroyd-B schemes}

\RRauthor{
S\'ebastien \textsc{Boyaval}, Tony \textsc{Lelièvre}, 
Claude \textsc{Mangoubi}\footnote{
also \textbf{Institute of Mathematics}, The Hebrew University, Jerusalem 91904 Israel
}
\\
{\footnotesize \textbf{CERMICS}, Ecole Nationale des Ponts et Chauss\'ees
   (ParisTech/Universit\'e Paris-Est),}\\ 
{\footnotesize6 \& 8 avenue Blaise Pascal,
   Cit\'e Descartes, 77455 Marne-la-Vall\'ee Cedex 2, France}\\
{\footnotesize and \textbf{MICMAC team-project}, INRIA,}\\
{\footnotesize Domaine de Voluceau, BP. 105-  Rocquencourt, 78153 Le Chesnay Cedex}\\
   \{boyaval, lelievre, mangoubi\}@cermics.enpc.fr
}

\authorhead{Boyaval \textit{et al.}}

\RRresume{
Dans cet article, nous analysons la stabilit\'e de divers sh\'emas num\'eriques
pour des mod\`eles diff\'erentiels de fluides visco\'elastiques.
Plus pr\'ecis\'ement, on consid\`ere le mod\`ele d'Oldroyd-B comme prototype
pour lequel on sait montrer qu'une \'energie libre est dissip\'ee, 
et on v\'erifie sous quelles hypoth\`eses le sh\'ema num\'erique
satisfait une propri\'et\'e semblable.
Parmi les sh\'emas num\'eriques analys\'es figurent des discr\'etisations
fond\'ees sur la reformulation logarithmique du syst\`eme Oldroyd-B
telle que propos\'ee par Fattal et Kupferman dans~\cite{fattal-kupferman-04},
reformulation qui a permis d'obtenir des r\'esultats num\'eriques
apparemment plus stables que les formulations usuelles dans certains cas tests~\cite{hulsen-fattal-kupferman-05}.
Notre analyse donne ainsi des pistes pour la compr\'ehension de ces observations num\'eriques.
}

\RRabstract{ 
In this article,  we analyze the stability of various numerical schemes
for  differential  models of  viscoelastic  fluids.  More precisely,  we
consider the prototypical Oldroyd-B model, for which a {\em free energy}
dissipation  holds,   and  we  show  under  which   assumptions  such  a
dissipation  is also  satisfied  for the  numerical  scheme.  Among  the
numerical schemes we analyze,  we consider some discretizations based on
the \emph{log-formulation}  of the  Oldroyd-B system proposed  by Fattal
and Kupferman in~\cite{fattal-kupferman-04}, which have been reported to
be numerically more stable than discretizations of the usual formulation
in   some   benchmark  problems~\cite{hulsen-fattal-kupferman-05}.   Our
analysis gives some tracks to understand these numerical observations.
}

\RRmotcle{Fluides visco\'elastiques ; Nombre de Weissenberg ; Stabilit\'e ; Entropie ;
M\'ethodes des \'el\'ements finis ; M\'ethode de Galerkin discontinue ; M\'ethode des caract\'eristiques}

\RRkeyword{Viscoelastic Fluids ; Weissenberg Number ; Stability ; Entropy ;
Finite Elements Methods ; Discontinuous Galerkin Method ; Characteristic Method}

\RRprojet{MICMAC}

\RRtheme{\THNum}

\URRocq

\begin{document}

\makeRR

\section{Introduction}

\subsection{The stability issue in numerical schemes for viscoelastic fluids}

An abundant literature has been discussing problems of stability for
numerical schemes discretizing models for \emph{viscoelastic fluids} for over twenty years
(see \cite{keunings-89,keunings-00,fattal-hald-katriel-kupferman-07,lee-xu-06} for a small sample).
Indeed, the numerical schemes for macroscopic constitutive equations
are known to suffer from instabilities when a parameter, the \emph{Weissenberg number}, increases.
Such an unstable behavior for viscoelastic fluids at moderately high Weissenberg ($\Wi \le 10$) 
is known to be unphysical.

Many possible reasons of that so-called \emph{high Weissenberg number problem} (HWNP)
have been identified \cite{keiller-92,kwon-leonov-95,sandri-99,fattal-kupferman-05}.
However, these results have not led yet to a complete understanding
of the numerical instabilities, despite some progress
\cite{fattal-kupferman-05,hulsen-fattal-kupferman-05}. Roughly speaking,
we can distinguish between three possible causes of the HWNP:
\begin{enumerate}
\item {\em Absence of stationary state:} In many situations (flow past a cylinder, 4:1 contraction), 
  the existence of a stationary state for viscoelastic models is still under investigation. 
  It may happen that the non-convergence of the numerical scheme is simply due to the fact that, 
  for the model under consideration, there exists no stationary state.
\item {\em Bad model:} More generally, the instabilities observed for the
  numerical scheme may originate at the continuous level, if the
  solution to the problem indeed blows up in finite time, or
  if it is not sufficiently regular to be well approximated in the discretization spaces.
\item {\em Bad numerical scheme:} It may also happen that the problem at
  the continuous level indeed admits a regular solution, and the instabilities
  are only due to the discretization method.
\end{enumerate}
In this paper, we focus on the third origin of instabilities, and we propose
a criterion to test the stability of numerical schemes. More precisely,
we look {\em under which conditions a numerical scheme does not bring
spurious free energy} in the system. We concentrate on the Oldroyd-B model, for which
a \emph{free energy dissipation} is known to hold at the continuous level 
(see Theorem~\ref{th:free-energy-classic} below and~\cite{hu-lelievre-07}) 
and we try to obtain a similar dissipation at the discrete level. 
It is indeed particularly important that no spurious free energy be brought to the system 
in the long-time computations, which are often used as a way to obtain the stationary state.

The Oldroyd-B system of equations is definitely not a good
physical model for dilute polymer fluids. In particular, it can be derived from a kinetic theory,
with dumbbells modeling polymer molecules that are unphysically assumed to be infinitely extensible.
But from the mathematical viewpoint, it is nevertheless a good first step into
the study of \emph{macroscopic constitutive equations} for viscoelastic fluids.
Indeed, it already contains mathematical difficulties common to most of the viscoelastic models, 
while its strict equivalence with a kinetic model allows for a deep
understanding of this set of equations. Let us also emphasize that the
free energy dissipation we use and the numerical schemes we
consider are not restricted to the Oldroyd-B model: they can be
generalized to many other models (like FENE-P for instance,
see~\cite{hu-lelievre-07}), so that we believe that our analysis can be
used as a guideline to derive ``good'' numerical schemes for many
macroscopic models for viscoelastic fluids.

\subsection{Mathematical setting of the problem}

We consider the Oldroyd-B model for dilute polymeric fluids in $d$-dimensional flows ($d=2,3$). 
Confined to an open bounded domain $\D\subset\R^d$, the fluid is governed by 
the following nondimensionalized system of equations:
\begin{equation}
\left\{
\begin{gathered}
\Re \brk{\pd{\vel}{t} + \vel\cdot\gvel} =  -\grad p + (1-\e)\Delta\vel +  \div\str, \\
\div\vel = 0, \\
\pd{\str}{t} + (\vel\cdot\grad)\str = (\gvel)\str + \str (\gvel)^T
-\frac{1}{\Wi}\str + \frac{\e}{\Wi} \Brk{\gvel + \gvel^T},
\end{gathered}
\right.
\label{eq:oldroyd-b-tau}
\end{equation}
where
 $\vel : (t,\x)\in[0,T)\times\D\to\vel(t,\x)\in\R^d$ is the velocity of the fluid, 
 $p : (t,\x)\in[0,T)\times\D\to p(t,\x)\in\R$ is the pressure and
 $\str : (t,\x)\in[0,T)\times\D\to\str(t,\x)\in\R^{d\times d}$ is the
 extra-stress tensor. The following parameters are dimensionless: the Reynolds number $\Re \in \R_+$, the Weissenberg number $\Wi \in \R_+^*$ and the elastic viscosity to total viscosity fraction $\e \in (0,1)$.

In all what follows, we assume for the sake of simplicity that the system~\eqref{eq:oldroyd-b-tau}
is supplied with homogeneous Dirichlet boundary conditions for the velocity $\vel$:
\begin{equation}\label{eq:Diric}
\vel=0 \text{ on $\partial\D$.}
\end{equation}
Therefore, we study the energy dissipation of the
equations~\eqref{eq:oldroyd-b-tau} as time goes, 
that is, the way $\brk{\vel,\str}$ converges to the stationary state $(0,0)$ (equilibrium) 
in the longtime limit $t \to \infty$. For free energy estimates under non-zero boundary conditions,
we refer to~\cite{jourdain-le-bris-lelievre-otto-06}.

Local-in-time existence results for the above problem have been proved in the bounded domain $[0,T)\times\D$
when the system is supplied with sufficiently smooth initial conditions $\vel(t=0)$ and $\str(t=0)$ (see~\cite{guillope-saut-90-a} and~\cite{fernandez-cara-guillen-ortega-02} for instance).
Moreover, global-in-time smooth solutions of the system~\eqref{eq:oldroyd-b-tau} are
known to converge exponentially fast to equilibrium in the sense defined
in~\cite{jourdain-le-bris-lelievre-otto-06}. 
Let us also mention the work of F.-H.~Lin, C.~Liu and P.W.~Zhang~\cite{lin-liu-zhang-05}
where, for Oldroyd-like models, local-in-time existence and uniqueness results are proven,
but also global-in-time existence and uniqueness results for small data. 
Notice that more general global-in-time results have been collected only 
for a mollified version of the Oldroyd-B system~\eqref{eq:oldroyd-b-tau} (see~\cite{barrett-schwab-suli-05}),
for another system close to~\eqref{eq:oldroyd-b-tau}, 
the co-rotational Oldroyd-B system (see~\cite{lions-masmoudi-00}),
or in the form of a Beale-Kato-Majda criterion when $\D=\R^3$ (see~\cite{kupferman-mangoubi-titi}).
Even though the question of the global-in-time existence 
for some solutions of the Oldroyd-B system~\eqref{eq:oldroyd-b-tau} is still out-of-reach,
it is possible to analyze global-in-time existence for solutions to
\emph{discretizations} of that system. This will be one of the output of
this article.

\subsection{Outline of the paper and results}

We will show that it is possible to build numerical schemes discretizing
the   Oldroyd-B  system~\eqref{eq:oldroyd-b-tau}--\eqref{eq:Diric}  such
that solutions  to those discretizations satisfy a  free energy estimate
similar                to                that                established
in~\cite{jourdain-le-bris-lelievre-otto-06,hu-lelievre-07}   for  smooth
solutions to  the continuous  equations.  Our approach  bears similarity
with~\cite{lozinski-owens-03},   where  the   authors   also  derive   a
discretization  that  preserves  an  energy estimate  satisfied  at  the
continuous    level,    and    with~\cite{lee-xu-06},   where    another
discretization   is   proposed  for   the   same   energy  estimate   as
in~\cite{lozinski-owens-03}.      Yet,      unlike     the     estimates
in~\cite{lozinski-owens-03,lee-xu-06},   our  estimate,   the  so-called
\emph{free              energy}             estimate             derived
in~\cite{jourdain-le-bris-lelievre-otto-06,hu-lelievre-07},  ensures the
long-time   stability  of  solutions.   As  mentioned   above,  long-time
computations are indeed often used to obtain a stationary state.

We also analyze discretizations of the log-formulation presented in~\cite{fattal-kupferman-04,fattal-kupferman-05}, 
where the authors suggest to rewrite the set of equations~\eqref{eq:oldroyd-b-tau} 
after mapping the (symmetric positive definite) \emph{conformation tensor}:
\begin{equation}\label{eq:conformation-tensor}
\strs = \I + \frac{\Wi}{\e}\str
\end{equation}
to its matrix logarithm:
\[
 \lstr = \ln \strs.
\]
In the following, we assume that:
\begin{equation}\label{eq:hyp_strs}
\text{ $\strs(t=0)$ is \emph{symmetric positive definite},}
\end{equation}
and it can be shown that this property is propagated in time
(see Lemma~\ref{lem:spd} below), so that $\lstr$ is indeed well defined.
The log-formulation insures, by construction, 
that the conformation tensor always remains symmetric positive definite,
even after discretization. This is not only an important physical characteristic of the Oldroyd-B model
but also an essential feature in the free energy estimates derived beneath. 
Besides, this log-formulation has indeed been observed to improve 
the stability of some computations~\cite{kwon-04,fattal-kupferman-05,hulsen-fattal-kupferman-05}. 
It is thus interesting to investigate whether the numerical success of this log-formulation 
may be related to the free energy dissipation.

The main outputs of this work are:
\begin{itemize}
\item One crucial feature of the numerical scheme to obtain free energy
  estimates is the discretization of the advection term
  $(\vel\cdot\grad)\str$ (or $(\vel\cdot\grad)\lstr$) in the equation on the extra-stress tensor. We
  will analyze below two types of discretization: characteristic method,
  and discontinuous Galerkin method (see Sections~\ref{sec:P0} and~\ref{sec:P1}).
\item To obtain free energy estimates, we will need the extra-stress
  tensor to be discretized in a (elementwise) discontinuous finite
  element space (see Sections~\ref{sec:P0} and~\ref{sec:P1}). 
  Yet, using a collocation method to compute bilinear forms,
  and with special care dedicated to the discretization of the advection terms,
  it would still be possible to use continuous finite element spaces,
  as it will be shown in a future work~\cite{barrett-boyaval}.
\item The existence of a solution to the numerical schemes that satisfy a free energy estimate 
  will be proved whatever the time step for the log-formulation in terms of $\lstr$, 
  while it will be shown under a CFL-like condition for the usual formulation in terms of $\str$ 
  (see Section~\ref{sec:stability}). This is due to the necessary positivity of the conformation tensor $\strs$
  for the free energy estimates in the usual formulation to hold,
  and may be related to the fact that the log-formulation has been reported to be more stable
  than the formulation in terms of $\str$ (see~\cite{hulsen-fattal-kupferman-05}),
  though the uniqueness of solutions is still not ensured in the long time limit and bifurcations may still occur.
  This result will also be re-investigated in the future work~\cite{barrett-boyaval}
  for a slightly modified formulation of the Oldroyd-B system of equations in terms of $\str$,
  somewhat closer to the usual formulation than the log-formulation.
\end{itemize}
Notice that we concentrate on stability issues. All the schemes we
analyze are of course consistent, but we do not study the order of
consistency of these schemes, neither the order of convergence. 
Again, this will be included in the future work~\cite{barrett-boyaval}
for the slightly modified formulation of the Oldroyd-B system of equations in terms of $\str$
mentioned above.

Let us now make precise how the paper is organized. In Section~\ref{sec:continuous}, we formally derive the
free energy estimates for the Oldroyd-B set of equations and
for its logarithm formulation, in the spirit of
\cite{hu-lelievre-07}. Then, Section~\ref{sec:construction} is devoted to the presentation of a 
finite element scheme (using piecewise constant approximations of
the conformation tensor and its log-formulation, and Scott-Vogelius
finite elements for the velocity and pressure),
  that is shown to satisfy a discrete free energy
estimate in Section~\ref{sec:P0}.  Some variants of this discretization
are also studied, still for piecewise constant stress tensor, and a summary of the requirements on the discretizations to satisfy a
free energy estimate is provided in Tables~\ref{tab:P0-summary} and~\ref{tab:finite element-summary}.
We show in Section~\ref{sec:P1} how to use an interpolation operator so as to adapt 
the previous results with piecewise linear approximations of the
conformation tensor and its log-formulation. Finally, in
Section~\ref{sec:stability}, we show how the previous stability results
can be used to prove long-time existence results for the discrete solutions. 
Some numerical studies are needed to illustrate this numerical
analysis, and this is a work in progress.

\subsection{Notation and auxiliary results}

In the following, we will make use of the usual notation: 
$L^2(\D)=\{f:\D \to \R, \int |f|^2 < \infty \}$, $H^1(\D)=\{f:\D \to \R, \int |f|^2 + |\nabla f|^2 < \infty \}$,  $H^2(\D)=\{f:\D \to \R, \int |f|^2 + |\nabla f|^2 + |\nabla^2 f|^2 < \infty \}$,  
$C([0,T))$ for continuous functions on $[0,T)$ and $C^1([0,T))$ for continuously differentiable functions on $[0,T)$.

We will denote by $\str:\strs$ the double contraction between rank-two
tensors (matrices) $\str$, $\strs \in \R^{d\times d}$:
$$\str:\strs=\tr(\str \strs^T)=\tr(\str^T \strs)=\sum_{1 \le i,j \le d} \str_{ij} \strs_{ij}.$$
Notice that if $\str$ is antisymmetric and $\strs$ symmetric, $\str:\strs=0$.

The logarithm of a positive definite diagonal matrix is a diagonal matrix with, on its diagonal, 
the logarithm of each entry. We define the logarithm of any symmetric positive definite matrix $\strs$ 
using a diagonal decomposition $\strs=R^T\Lambda R$ of $\strs$
with $R$ an orthogonal matrix and $\Lambda$ a positive definite diagonal matrix:
\begin{equation}\label{eq:log-definition}
\ln \strs = R^T \, \ln\Lambda \, R.
\end{equation}
Although the diagonal decomposition of $\strs$ is not unique, \eqref{eq:log-definition} uniquely defines $\ln\strs$.
The matrix logarithm bijectively maps the set of 
symmetric positive definite matrices with real entries $\mathcal{S}_+^*(\R^{d\times d})$
to the vector subspace $\mathcal{S}(\R^{d\times d})$ of symmetric real matrices,
where it is exactly the reciprocal function of the matrix exponential.

We will make use of the following simple algebraic formulae, which are proved in Appendix \ref{ap:alg_spd} and \ref{ap:dtrdt}.
\begin{lemma}\label{lemma:alg_spd} 
Let $\strs$ and $\str$ be two symmetric positive definite matrices. We have:
\begin{equation}\label{eq:trln}
\tr \ln \strs = \ln \det \strs,
\end{equation}
\begin{equation}\label{eq:tr-ln}
\strs - \ln \strs - \I \text{ is symmetric positive semidefinite and thus }\tr(\strs - \ln \strs - \I) \ge 0, 
\end{equation}
\begin{equation}\label{eq:tr-plus_inv}
\strs + \strs^{-1}  - 2 \I \text{ is symmetric positive semidefinite and thus } \tr(\strs + \strs^{-1}  - 2 \I) \ge 0, 
\end{equation}
\begin{equation}\label{eq:AB}
\tr(\strs\str) = \tr(\str\strs) \ge 0,
\end{equation}
\begin{equation}\label{eq:traceAB}
\tr\brk{ (\strs-\str) \str^{-1} } = \tr(\strs \str^{-1}-\I) \ge \ln\det(\strs \str^{-1})= \tr\brk{ \ln \strs - \ln \str},
\end{equation}
\begin{equation}\label{eq:traceAB+}
\tr\brk{\brk{ \ln \strs - \ln \str} \strs } \geq \tr\brk{\strs - \str}.
\end{equation}
\end{lemma}

We will also use the usual Jacobi's formula:
\begin{lemma}\label{lemma:dtrdt}
For any symmetric positive definite matrix $\strs(t) \in C^1\left([0,T)\right)$, we have $\forall t \in [0,T)$:
\begin{equation}\label{eq:tracelog}
\brk{\deriv{}{t}\strs}:\strs^{-1} = \tr\brk{\strs^{-1}\deriv{}{t}\strs} = \deriv{}{t}\tr(\ln \strs)
,
\end{equation}
\begin{equation}\label{eq:trace}
\brk{\deriv{}{t}\ln \strs}:\strs = \tr\brk{\strs\deriv{}{t}\ln \strs} = \deriv{}{t}\tr \strs
.
\end{equation}
\end{lemma}


\section{Formal free energy estimates at the continuous level}
\label{sec:continuous}

We are going to derive free energy estimates
for two formulations of the Oldroyd-B system
 in Theorems~\ref{th:free-energy-classic} and~\ref{th:free-energy-log}.
An important corollary to these theorems is
the exponential convergence of the solutions to equilibrium
in the longtime limit.
In all that follows, we assume that $(\vel,p,\str)$ is a sufficiently \emph{smooth}
solution of problem \eqref{eq:oldroyd-b-tau} so that all the subsequent
computations are valid. For example, the following regularity is sufficient:
\begin{equation}\label{eq:regularity}
(\vel,p,\str) \in 
\brk{C^1\brk{[0,T),H^2(\D)}}^d \times 
\brk{C^0\brk{[0,T),H^1(\D)}} \times 
\brk{C^1\brk{[0,T),C^1(\D)}}^{d\times d},
\end{equation}
where we denote, for instance by $\brk{C^1(\D)}^d$ a vector field of dimension $d$ with $C^1(\D)$ components.

\subsection{Free energy estimate for the Oldroyd-B system}

\subsubsection{Conformation-tensor formulation of the Oldroyd-B system}

Recall that the \emph{conformation} tensor $\strs$ is defined from 
the \emph{extra-stress} tensor $\str$ through the following bijective mapping:
\[
\str = \frac{\e}{\Wi}\brk{\strs - \I}.
\]
With this mapping, it is straightforward to bijectively map the solutions of system~\eqref{eq:oldroyd-b-tau} 
with those of the following system:
\begin{equation}
\left\{
\begin{gathered}
\Re \brk{\pd{\vel}{t} + \vel\cdot\gvel} =  -\grad p + (1-\e)\Delta\vel +  \frac{\e}{\Wi}\div\strs, \\
\div\vel = 0, \\
\pd{\strs}{t} + (\vel\cdot\grad)\strs = (\gvel)\strs + \strs
(\gvel)^T -\frac{1}{\Wi}(\strs-\I).
\end{gathered}
\right.
\label{eq:oldroyd-b-sigma}
\end{equation}
Notice that with such an affine mapping, the solution $\strs$ to system~\eqref{eq:oldroyd-b-sigma}
has the same regularity than $\str$ solution to system~\eqref{eq:oldroyd-b-tau},
which is that assumed in~\eqref{eq:regularity} for the following manipulations.

\subsubsection{A free energy estimate}

Let us first recall a free energy estimate derived in~\cite{jourdain-le-bris-lelievre-otto-06, hu-lelievre-07}.
The free energy of the fluid is defined as the sum of 
two terms as follows:
\begin{equation} \label{def:free-energy-classic}
F(\vel,\strs) = \frac{\Re}{2}\intd|\vel|^2 + \frac{\e}{2\Wi}\intd\tr(\strs - \ln\strs - \I).
\end{equation}

The \emph{kinetic} term $\displaystyle{\intd|\vel|^2}$ is always non negative. 
Besides, we have the following lemma (see~Appendix~\ref{ap:spd} or~\cite{hulsen-90} for a proof):
\begin{lemma}\label{lem:spd}
Let $\strs\in\brk{C^1\brk{[0,T),C^1(\D)}}^{d\times d}$ be a smooth solution to the system~\eqref{eq:oldroyd-b-sigma}.
Then, if the initial condition $\strs(t=0)$ is symmetric positive definite (everywhere in $\D$), 
the solution $\strs(t)$ remains so at all times $t \in [0,T)$ and for all $\x\in\D$.
In particular, the matrix $\strs(t)$ is invertible.
\end{lemma}
From Lemma~\ref{lem:spd} and the equation~\eqref{eq:tr-ln}, the
\emph{entropic} term $\displaystyle \intd\tr(\strs - \ln\strs - \I)$ is
thus well defined and non negative, provided $\strs(t=0)$ is symmetric positive definite.

The free energy is an interesting quantity to characterize the long-time asymptotics of the solutions, 
and thus the stability of system \eqref{eq:oldroyd-b-sigma}. 
{\em A priori} estimates using the free energy are presented
in~\cite{jourdain-le-bris-lelievre-otto-06} for micro-macro models 
(such as the Hookean or the FENE dumbbell models) 
and in~\cite{hu-lelievre-07} for macroscopic models 
(such as the Oldroyd-B or the FENE-P models). Similar considerations can
be found in the physics literature about thermodynamic theory for
viscoelastic models (see~\cite{leonov-92,beris-edwards-94,oettinger-05,wapperom-hulsen-98}).

For the sake of consistency, we recall results from~\cite{hu-lelievre-07}:
\begin{theorem} \label{th:free-energy-classic}
Let $(\vel,p,\strs)$ be a smooth solution to system
\eqref{eq:oldroyd-b-sigma} supplied with homogeneous Dirichlet
boundary conditions for $\vel$, and with symmetric positive definite initial condition
$\strs(t=0)$. The free energy satisfies:
\begin{equation}
\deriv{}{t}F(\vel,\strs) +(1-\e)\intd|\gvel|^2 + \frac{\e}{2\Wi^2}\intd\tr(\strs
+ \strs^{-1} - 2\I)=0.
\label{eq:free-equality-classic}
\end{equation}
{F}rom this estimate, we get that $F(\vel,\strs)$ decreases exponentially fast in time to zero.
\end{theorem}

\begin{proof}[Proof of Theorem \eqref{th:free-energy-classic}]
Let $(\vel,p,\strs)$ be a smooth solution to system \eqref{eq:oldroyd-b-sigma},
with symmetric positive definite initial condition $\strs(t=0)$. We first compute the inner product of
the Navier-Stokes equation with the velocity:
\begin{equation}
\frac{\Re}{2}\deriv{}{t}\intd |\vel|^2 = -(1-\e) \intd |\gvel|^2 -
\frac{\e}{\Wi} \intd \gvel : \strs \label{eq:energy1-sigma}.
\end{equation}
Then, taking the trace of the evolution equation for the conformation tensor, we obtain:
\begin{equation}
\deriv{}{t} \intd \tr \strs = 2\intd \gvel : \strs
-\frac{1}{\Wi}\intd \tr(\strs - \I).
\label{eq:energy2-sigma}
\end{equation}
Last, remember that smooth solutions $\strs$ are invertible matrices (Lemma \ref{lem:spd}). Thus, contracting the evolution equation for $\strs$ with $\strs^{-1}$, we get:
\begin{equation}
\intd
\brk{\pd{}{t}\strs+(\vel\cdot\grad)\strs}:\strs^{-1}=\intd\tr(\gvel)-\frac{1}{\Wi}\intd\tr(\I-\strs^{-1}).
\label{eq:energy3prime-sigma}
\end{equation}
Using~\eqref{eq:tracelog} with $\strs \in C^1\brk{\D\times[0,T),\mathcal{S}^\star_+(\R^{d\times d})}$, we find:
\[
\intd \brk{\pd{}{t}\strs+(\vel\cdot\grad)\strs}:\strs^{-1} = \intd
\brk{\pd{}{t}+\vel\cdot\grad}\tr(\ln\strs),
\]
which can be combined with \eqref{eq:energy3prime-sigma} to get, using
$\tr(\gvel)=\div\vel=0$ and $\vel=0$ on $\partial\D$:
\begin{equation} \label{eq:energy3-sigma}
\deriv{}{t} \intd
\tr\,\ln\strs = \frac{1}{\Wi}\intd\tr(\strs^{-1}-\I).
\end{equation}
We now combine \eqref{eq:energy1-sigma} $+\frac{\e}{2\Wi}\times$
\eqref{eq:energy2-sigma} $-\frac{\e}{2\Wi}\times$ \eqref{eq:energy3-sigma}
to obtain \eqref{eq:free-equality-classic}:
\begin{equation*}
\deriv{}{t}\Brk{\frac\Re2\intd|\vel|^2 +
  \frac{\e}{2\Wi}\intd\tr(\strs-\ln\strs - \I)} 
+(1-\e)\intd|\gvel|^2 + \frac{\e}{2\Wi^2}\intd\tr(\strs + \strs^{-1} - 2\I)=0.
\end{equation*}

Since, by~\eqref{eq:tr-plus_inv}, we have $\tr(\strs + \strs^{-1} - 2\I) \ge 0$, 
then $F(\vel,\strs)$ decreases in time. 
Moreover, by~(\ref{eq:tr-ln}) applied to $\strs^{-1}$, we have 
$\tr(\strs - \ln \strs - \I) \le \tr(\strs + \strs^{-1} - 2\I)$.
So, using the Poincar\'e inequality which states that there exists a
constant $C_P$ depending only on $\D$ such that, for all $\vel\in H^1_0(\D)$,
$$\int_\D |\vel|^2 \leq C_P \int_\D |\gvel|^2,$$
we finally obtain that $F(\vel,\strs)$ goes exponentially fast to $0$. Indeed, we
have from~\eqref{eq:free-equality-classic}:
\begin{align*}
\deriv{}{t}F(\vel,\strs)
&\leq -\frac{1-\e}{C_P} \int_\D |\vel|^2 -
\frac{\e}{2\Wi^2}\intd\tr(\strs + \strs^{-1} - 2\I),\\
&\leq - \min\left( \frac{2(1-\e)}{\Re \ C_P},\frac{1}{\Wi}\right) F(\vel,\strs),
\end{align*}
so that, by a direct application of Gronwall's lemma, we get:
$$F(\vel,\strs) \le F(\vel(t=0),\strs(t=0)) 
 \exp \left( \min\left( \frac{2(1-\e)}{\Re \ C_P},\frac{1}{\Wi}\right) t \right) \,.$$
\end{proof}

\subsection{Free energy estimate for the log-formulation of the Oldroyd-B system}

\subsubsection{Log-formulation of the Oldroyd-B system}

Let us now introduce the log-formulation proposed in~\cite{fattal-kupferman-04}. 
We want to map solutions of the system~\eqref{eq:oldroyd-b-sigma} 
with solutions of another system of equations where 
a partial differential equation for the logarithm of the conformation tensor
is substituted to the Oldroyd-B partial differential equation for the conformation tensor $\strs$.

In order to obtain a constitutive equation in terms of $\lstr=\ln\strs$,
following~\cite{fattal-kupferman-04},
we make use of the following decomposition of the deformation tensor
$\gvel \in \R^{d \times d}$ (see Appendix~\ref{app:decomposition} for a proof):
\begin{lemma} \label{lemma:decomposition}
For any matrix $\gvel$ and any symmetric positive definite matrix $\strs$ in $\R^{d \times d}$, 
there exist in $\R^{d \times d}$ 
two antisymmetric matrices $\Om$, $\N$ and a symmetric matrix $\B$ that commutes with $\strs$, such that:
\begin{equation} \label{eq:velocity_decomposition}
\gvel = \Om + \B + \N\strs^{-1}.
\end{equation}
Moreover, we have $\tr\gvel=\tr\B$.
\end{lemma}

We now proceed to the change of variable $\lstr = \ln \strs$.
The system~\eqref{eq:oldroyd-b-sigma} then rewrites
(see~\cite{fattal-kupferman-04} for a proof):
\begin{equation}
\left\{
\begin{gathered}
\Re \brk{\pd{\vel}{t} + \vel\cdot\gvel} =  -\grad p + (1-\e)\Delta\vel +  \frac{\e}{\Wi}\div  \elstr,\\
\div\vel = 0, \\
\pd{\lstr}{t} + (\vel\cdot\grad)\lstr = \Om\lstr - \lstr\Om + 2\B + \frac{1}{\Wi}(e^{-\lstr}-\I).
\end{gathered}
\right.
\label{eq:oldroyd-b-log}
\end{equation}
It is supplied with unchanged initial and boundary conditions for $\vel$,
plus the initial condition $\lstr(t=0) = \ln\strs(t=0)$ for the log-conformation tensor.

\subsubsection{Reformulation of the free energy estimate}

A result similar to Theorem \ref{th:free-energy-classic} can be obtained for system \eqref{eq:oldroyd-b-log}, 
where the free energy is written in terms of $\lstr$ as:
\begin{equation} \label{def:free-energy-log}
F(\vel,e^\lstr) =\frac{\Re}{2}\intd|\vel|^2 +
\frac{\e}{2\Wi}\intd\tr(\elstr - \lstr - \I).
\end{equation}
The following theorem then holds :
\begin{theorem} \label{th:free-energy-log}
Let $(\vel,p,\lstr)$ be a smooth solution to system \eqref{eq:oldroyd-b-log}
supplied with homogeneous Dirichlet boundary conditions for $\vel$.
The free energy satisfies:
\begin{equation}\label{eq:free-equality-log}
\deriv{}{t}F(\vel,e^\lstr) +(1-\e)\intd|\gvel|^2 + \frac{\e}{2\Wi^2}\intd\tr(\elstr + \mlstr - 2\I)=0.
\end{equation}
{F}rom this estimate, we get that $F(\vel,e^\lstr)$ decreases exponentially fast in time
to zero.
\end{theorem}

\begin{proof}[Proof of Theorem \eqref{th:free-energy-log}]
The proof of this theorem mimicks the proof of Theorem
\ref{th:free-energy-classic}. We go over the steps of the proof, and point out the differences
with the previous case. Let $(\vel,p,\lstr)$ be a smooth solution to
\eqref{eq:oldroyd-b-log}.

{F}rom the inner product of the momentum conservation equation in \eqref{eq:oldroyd-b-log} with the velocity $\vel$,
we obtain:
\begin{equation}
\label{eq:energy1-log}
\frac{\Re}{2}\deriv{}{t}\intd |\vel|^2 = -(1-\e) \intd |\gvel|^2 -
\frac{\e}{\Wi} \intd \gvel : \elstr ,
\end{equation}
which is equivalent to \eqref{eq:energy1-sigma}. 
Taking the trace of the evolution equation for the conformation tensor, we get:
\begin{equation}
\label{eq:energy2-log}
\deriv{}{t} \intd \tr \lstr = \frac{1}{\Wi}\intd\tr(\mlstr -\I),
\end{equation}
which is equivalent to \eqref{eq:energy3-sigma}.
Contracting the evolution equation for $\lstr$ with $\elstr$ and using~\eqref{eq:trace} with $\lstr=\ln\strs$,
we rewrite the first term of this inner product:
\[
\brk{\pd{\lstr}{t}+\vel\cdot\grad\lstr}:\elstr =
\brk{\pd{}{t}+\vel\cdot\grad}\tr\elstr.
\]
Recall that the decomposition \eqref{eq:velocity_decomposition} of $\gvel$ allows to rewrite the second term:
\begin{equation}\label{eq:uB}
\gvel:\elstr = \Om:\elstr + \B:\elstr + (\N\mlstr):\elstr = \B:\elstr\,,
\end{equation}
where we have used the symmetry of $\elstr$ and the antisymmetry of $\Om$ and $\N$.
Then, notice that, since $\lstr$ and $\elstr$ commute, we have:
\begin{align}
(\Om\lstr-\lstr\Om):\elstr&=\tr(\Om\lstr\elstr) - \tr(\lstr\Om\elstr),\nonumber \\
&=\tr(\Om\lstr\elstr) -\tr(\Om\lstr\elstr), \nonumber \\
&=0\,,
\label{eq:omegapsi0}
\end{align}
we finally obtain an equation equivalent to \eqref{eq:energy2-sigma}:
\begin{equation} \label{eq:energy3-log}
\deriv{}{t} \intd \tr \elstr = 2\intd \gvel : \elstr
-\frac{1}{\Wi}\intd \tr(\elstr - \I) .
\end{equation}
It is noticeable that in this proof, we made no use of the positivity of $\strs=\elstr$,
in contrast to the proof of Theorem \ref{th:free-energy-classic}.

The combination \eqref{eq:energy1-log} $-\frac{\e}{2\Wi}\times$
\eqref{eq:energy2-log} $+\frac{\e}{2\Wi}\times$ \eqref{eq:energy3-log} gives \eqref{eq:free-equality-log}:
\begin{equation}
\deriv{}{t}\Brk{\frac\Re2\intd|\vel|^2
  +\frac{\e}{2\Wi}\intd\tr(\elstr-\lstr - \I)} 
+(1-\e)\intd|\gvel|^2 + \frac{\e}{2\Wi^2}\intd\tr(\elstr+\mlstr-2\I)=0.
\end{equation}
This is exactly equivalent to \eqref{eq:free-equality-classic}. As in
the proof of Theorem \ref{th:free-energy-classic}, we then obtain that $F(\vel,e^\lstr)$ decreases exponentially fast in time to zero.
\end{proof}



\section{Construction of numerical schemes with Scott-Vogelius elements for $(\velh,\ph)$}\label{sec:construction}

We would now like to build numerical integration schemes for both systems
of equations~\eqref{eq:oldroyd-b-sigma} and~\eqref{eq:oldroyd-b-log}
that respectively preserve the dissipation properties of Theorems \ref{th:free-energy-classic}
and \ref{th:free-energy-log} for discrete free energies similar to \eqref{def:free-energy-classic}
and \eqref{def:free-energy-log}. We first present discretizations 
that allow for a simple and complete exposition of our reasoning in order to derive discrete free energy estimates.
Possible extensions will be discussed in the Sections~\ref{sec:need} and~\ref{sec:P1}.

\subsection{Variational formulations of the problems}

To discretize~\eqref{eq:oldroyd-b-sigma} and~\eqref{eq:oldroyd-b-log} in space 
using a finite element method,
we first write variational formulations for~\eqref{eq:oldroyd-b-sigma} and~\eqref{eq:oldroyd-b-log} 
that are satisfied by smooth solutions of the previous systems. Smooth solutions $(\vel,p,\strs)$ and $(\vel,p,\lstr)$ 
to system~\eqref{eq:oldroyd-b-sigma} and~\eqref{eq:oldroyd-b-log}
respectively satisfy the variational formulations:
\begin{equation}\label{eq:form-sigma}
 \begin{split}
0 = \intd \Re\brk{\pd{\vel}{t} +
\vel\cdot\gvel}\cdot\bv +
(1-\e)\gvel:\grad\bv - p\div\bv \\
+\frac{\e}{\Wi}\strs:\grad\bv + q\div \vel \\
+ \brk{\pd{\strs}{t}+\vel\cdot\grad\strs}:\bphi - ((\gvel)\strs +
\strs(\gvel)^T):\bphi + \frac{1}{\Wi}(\strs-\I):\bphi,
\end{split}
\end{equation}
and
\begin{equation}\label{eq:form-log}
\begin{split}
0 = \intd \Re\brk{\pd{\vel}{t} +
\vel\cdot\gvel}\cdot\bv +
(1-\e)\gvel:\grad\bv - p\div\bv\\
+\frac{\e}{\Wi}\elstr:\grad\bv + q\div \vel \\
+ \brk{\pd{\lstr}{t}+\vel\cdot\grad\lstr}:\bphi - (\Om\lstr -
\lstr\Om):\bphi - 2\B:\bphi - \frac{1}{\Wi}(\mlstr-\I):\bphi,
\end{split}
\end{equation}
for all sufficiently regular test function $(\bv,q,\bphi)$.

In this variational framework, we recover 
the free energy estimates~\eqref{eq:free-equality-classic}
(respectively~\eqref{eq:free-equality-log}) 
using the test functions $\brk{\vel,p,\frac{\e}{2\Wi}(\I-\strs^{-1})}$
(respectively $\brk{\vel,p,\frac{\e}{2\Wi}(\elstr-\I)}$)
in~\eqref{eq:form-sigma} (respectively~\eqref{eq:form-log}).


\subsection{Numerical schemes with Scott-Vogelius finite elements for $(\velh,\ph)$}
\label{sec:discrete}

Using the Galerkin discretization method,
we now want to build variational numerical integration schemes
that are based on the variational formulations~\eqref{eq:form-sigma} and~\eqref{eq:form-log}
using finite-dimensional approximations of the solution/test spaces.
We will then show in the next Section~\ref{sec:P0}
that solutions to these schemes satisfy discrete free energy estimates
which are equivalent to those in Theorems~\ref{th:free-energy-classic} and~\ref{th:free-energy-log}.

First, the time interval $[0,T)$ is split into $N_T$ intervals $[t^n,t^{n+1})$ 
of constant size $\dt = \frac{T}{N_T}$, with $t^n=n\dt$ for $n = 0 \dots N_T$.
 For all $n = 0 \dots N_T-1$, we denote by $(\velh^{n},\ph^{n},\strh^{n})$ 
(resp.~$(\velh^{n},\ph^{n},\lstrh^{n})$),
the value at time $t_n$ of the discrete solutions $(\velh,\ph,\strh)$ 
(resp.~$(\velh,\ph,\lstrh)$) in finite element spaces.

In all the following sections, we will assume that the domain $\D$ is polyhedral.
We define a conformal mesh $\mathcal{T}_h$ built from a tesselation of the domain $\D$,
$$ \mathcal{T}_h = \mathop{\cup}_{k=1}^{N_K} K_k \,, $$
made of $N_K$ simplicial elements $K_k$ and $N_D$ nodes at the internal vertices.
We denote by $h_{K_k}$ the diameter of the element $K_k$ and assume that the mesh is uniformly regular, 
with maximal diameter $h \ge \max_{1\le k\le N_K} h_{K_k}$. 
For each element $K_k$ of the mesh $\mathcal{T}_h$, 
we denote by $\bn_{K_k}$ the outward unitary normal vector to element $K_k$, 
defined on its boundary $\partial K_k$.  
We also denote by $\{E_j|j=1,\dots,N_E\}$ the internal edges of the mesh $\mathcal{T}_h$
when $d=2$, or the faces of volume elements when $d=3$ 
(also termed as ``edges'' for the sake of simplicity in the following). 

For the velocity-pressure field $(\velh,\ph)$, 
we choose the mixed finite element space $(\Pt)^d \times \Pod$ of Scott-Vogelius \cite{scott-vogelius-85}, where:
\begin{itemize}
\item by $\velh\in(\Pt)^d$ we mean that $\velh$ is a vector field with entries over $\D$
that are continuous polynomials of maximal degree $2$,
\item and by $\ph\in\Pod$ we mean that $\ph$ is a scalar field with entries over ${\cal T}_h$
that are piecewise continuous polynomials of maximal degree $1$ (thus discontinuous over $\D$).
\end{itemize}
This choice is very convenient to establish the free-energy estimates at the discrete level. 
As mentioned earlier, other choices will be discussed in Section~\ref{sec:need}. 
For general meshes, this finite element does not satisfy the Babu\v{s}ka-Brezzi inf-sup condition. 
However, for meshes built using a particular process based on a first mesh of macro-elements, 
this mixed finite element space is known to satisfy the Babu\v{s}ka-Brezzi inf-sup condition
(this is detailed in \cite{arnold-qin-92} for instance). 
The interest of this finite element is that the velocity field is divergence-free:
\begin{equation} \label{eq:divu=0}
 \div \velh(\x) = 0,\, \forall\x\in{\D} ,
\end{equation}
because $\div \velh \in \Pod$ can be used as a test function for the
pressure field in the weak formulation of the incompressibility
constraint $\int_\D \div \velh q_h = 0$.

For the approximation of $\strh$ and $\lstrh$, we use \emph{discontinuous} finite elements to derive the free energy estimates. 
For simplicity, we first consider piecewise constant approximations of $\strh$ and $\lstrh$ 
in Sections~\ref{sec:construction} and~\ref{sec:P0}. 
In Section~\ref{sec:P1}, we will come back to this assumption 
and discuss the use of higher order finite element spaces for $\strh$ and $\lstrh$.
All along this work, we denote by $\strh \in (\Pz)^{\dd}$ the fact that 
the symmetric-tensor field $\strh$ is discretized using a $\dd$-dimensional so-called \emph{stress} field,
which stands for the entries in $\Pz$ of a symmetric $d \times d$-dimensional tensor field, 
thus enforcing the symmetry in the discretization.

The advection terms $\vel \cdot \grad \strs$ and $\vel \cdot \grad \lstr$ will be discretized
either through a \emph{characteristic} method in the spirit of~\cite{pironneau-82,baranger-machmoum-97,wapperom-keunings-legat-00},
or with the \emph{discontinuous Galerkin} (DG) method in the spirit of~\cite{hulsen-fattal-kupferman-05}. 
Notice already that the characteristic method requires the velocity field to be more regular 
 than the discontinuous Galerkin method in order to define the flow associated with the vector field~$\velh$.

For the discontinuous Galerkin method, we will need the following notation. 
Let $E_j$ be some internal edge in the mesh $\mathcal{T}_h$.
To each edge $E_j$, we associate a unitary orthogonal vector $\bn\equiv\bn_{E_j}$,
whose orientation will not matter in the following.
Then, for a given velocity field $\velh$ in $\D$ that is well defined on the edges, 
for any variable $\bphi$ in $\D$ and any interior point $\x$ to the edge $E_j$,
we respectively define the downstream and upstream values of $\bphi$ by:
\begin{equation}\label{eq:sign-convention}
\begin{aligned}
\bphi^+(\x) & = \lim_{\delta \rightarrow 0^+} \bphi(\x+\delta\,\velh(\x)), \\
\bphi^-(\x) & = \lim_{\delta \rightarrow 0^-}
\bphi(\x+\delta\,\velh(\x)).
\end{aligned}
\end{equation}
We denote by $ \jump{\bphi}(\x) = \bphi^+(\x) - \bphi^-(\x) $ the jump of $\bphi$ over the edge $E_j$
and by $\BRK{\bphi}(\x) = \frac{\bphi^+(\x) + \bphi^-(\x)}{2}$ the mean value over the edge.
Then, one can easily check the following formula for any function $\phi$:
\begin{equation}\label{eq:jump}
\sum_{E_j} \int_{E_j} |\velh\cdot\bn| \jump{\phi}
= -\sum_{K_k} \int_{\partial K_k} \brk{\velh \cdot \bn_{K_k}} \phi.
\end{equation}

Let us now present in the next section the discrete variational
formulations we will consider.

\begin{remark}
In what follows, we do not consider the possible instabilities occurring when advection dominates diffusion 
in the Navier-Stokes equation for the velocity field $\velh$. Indeed, in practice, one typically considers small Reynolds number flows for polymeric fluids, so that we are in a regime where such instabilities are not observed.

Moreover, we also assume that $0\leq\e< 1$ so that there is no problem of compatibilities between the discretization space for the velocity and for the stress (see~\cite{bonvin-picasso-stenberg-01} for more details).
\end{remark}

\subsection{Numerical schemes with $\strh$ piecewise constant}
\label{sec:P0-strh}

Variational formulations of the discrete problem write, for all $n = 0 \dots N_T-1$, as follows:

\paragraph{With the {\em characteristic} method:}
For a given $(\velhn,\ph^{n},\strhn)$, 
find $(\velh^{n+1},\ph^{n+1},\strh^{n+1}) \in (\Pt)^d\times\Pod\times(\Pz)^\dd$ such that,
for any test function $(\bv,q,\bphi) \in (\Pt)^d\times\Pod\times(\Pz)^\dd$,
\begin{equation} \label{eq:P0-charac}
\begin{split}
0 = \intd & \Re\brk{\frac{\velhnp-\velhn}{\dt} + \velhn\cdot\gvelhnp}\cdot\bv 
\\
& - \phnp\div\bv + q\div \velhnp+ (1-\e)\gvelhnp:\grad\bv
  +\frac{\e}{\Wi}\strhnp:\grad\bv
\\
&+ \brk{  \frac{\strhnp-\strhn \circ X^n(t^n) }{\dt} }:\bphi
- \brk{ (\gvelhnp)\strhnp +\strhnp(\gvelhnp)^T }:\bphi
\\
&+ \frac{1}{\Wi}(\strhnp-\I):\bphi.
\end{split}
\end{equation}
This problem is supplied with an initial condition $(\velh^0,\ph^0,\strh^0) \in (\Pt)^d\times\Pod\times(\Pz)^\dd$.
\\
The function $X^n(t): x \in \D \mapsto X^n(t,x) \in \D$ is the
``backward'' flow associated with the velocity field $\velh^n$ and
satisfies, for all $x \in \D$,
\begin{equation} \label{eq:flow-P0-charac}
\left\{
\begin{array}{l}
 \deriv{}{t} X^n(t,x) = \velh^n ( X^n(t,x) ), \quad \forall t \in [t^n,t^{n+1}],\\
 X^n(t^{n+1},x) = x. 
\end{array}
\right.
\end{equation}

\paragraph{With the {\em discontinuous Galerkin} method:}
For a given $(\velhn,\ph^{n},\strhn)$, 
find $(\velh^{n+1},\ph^{n+1},\strh^{n+1}) \in (\Pt)^d\times\Pod\times(\Pz)^\dd$ 
such that, for any test function $(\bv,q,\bphi) \in (\Pt)^d\times\Pod\times(\Pz)^\dd$,
\begin{equation} \label{eq:P0-DG}
\begin{split}
0 =  \sum_{k=1}^{N_K} \int_{K_k} & \Re\brk{\frac{\velhnp-\velhn}{\dt} + \velhn\cdot\gvelhnp}\cdot\bv 
\\
& - \phnp\div\bv + q\div \velhnp+ (1-\e)\gvelhnp:\grad\bv
  +\frac{\e}{\Wi}\strhnp:\grad\bv
\\
&+ \brk{  \frac{\strhnp-\strhn}{\dt} }:\bphi
- \brk{ (\gvelhnp)\strhnp +\strhnp(\gvelhnp)^T }:\bphi
\\
&+ \frac{1}{\Wi}(\strhnp-\I):\bphi
\\
+ \sum_{j=1}^{N_E} &\int_{E_j}  \Abs{\velhn\cdot\bn} \jump{\strhnp} : \bphi^+ .
\end{split}
\end{equation}
Since $\strh \in (\Pz)^\dd$ is discontinuous, 
we have discretized the advection term for $\strh$ with a sum of jumps similar to the usual upwind technique, 
where $\bphi^+=\left(\frac12\jump{\bphi}+\BRK{\bphi}\right)$ (see~\cite{ern-guermond-04,hulsen-fattal-kupferman-05}).

\begin{remark}
In all the following, we assume that, when using the characteristic method:
\begin{itemize}
\item the characteristics are exactly integrated,
\item and the integrals involving the backward flow $X^n$ are exactly computed. 
\end{itemize}
We are aware of the fact that these assumptions are strong, 
and that numerical instabilities may be induced by bad integration schemes. 
Hence, considering the lack for an analysis of those integration schemes for the characteristics in the present study,
our analysis of discontinuous Galerkin discretizations of the advection terms 
may seem closer to the real implementation than that of the discretizations using the characteristic method.
\end{remark}

\begin{remark}
Notice that the numerical schemes we propose are nonlinear 
due to the implicit terms corresponding to the discretization of the upper-convective derivative 
$(\gvel)\strs+\strs(\gvel)^T $ (resp. $\Omega \lstr-\lstr\Omega$). 
In practice, this nonlinear system can be solved by fixed point procedures, 
either using the values at the previous time step as an initial guess, 
or using a predictor obtained by solving another scheme where the nonlinear terms are explicited.
\end{remark}

\subsection{Numerical schemes with $\lstrh$ piecewise constant}
\label{sec:log-P0}

We now show how to discretize the variational log-formulation similarly as above.
For this, we will need the following elementwise decomposition of the velocity gradient
(see Lemma~\ref{lemma:decomposition} above):
\begin{equation}\label{eq:discrete-decomposition}
\gvelhnp = \Omh^{n+1} + \B_h^{n+1} + \Nh^{n+1}\emstrhnp.
\end{equation}
Moreover, for the decomposition \eqref{eq:discrete-decomposition} with $\vel \in (\Pt)^d$, 
we will need the following Lemma~\ref{lemma:discrete-decomposition} for $k=1$, which is proved in Appendix~\ref{app:decomposition}:
\begin{lemma}\label{lemma:discrete-decomposition}
Let $\gvelhnp \in (\mathbb{P}_{k,disc})^{d\times d}$ for some $k \in \mathbb{N}$.
Then, for any symmetric positive definite matrix $\elstrhnp \in (\Pz)^\dd$,
there exist two antisymmetric matrices $\Omh^{n+1},\Nh^{n+1} \in (\mathbb{P}_{k,disc})^{\frac{d(d-1)}{2}}$ 
and a symmetric matrix $\B_h^{n+1} \in (\mathbb{P}_{k,disc})^\dd$ that commutes with $\elstrhnp$,
such that the matrix-valued function $\gvelhnp$ can be decomposed pointwise as:
$ \gvelhnp = \Omh^{n+1} + \B_h^{n+1} + \Nh^{n+1}\emstrhnp. $
\end{lemma}

Variational formulations of the discrete problem write, for all $n = 0 \dots N_T-1$, as follows:
\paragraph{With the \emph{characteristic} method:}
For a given $(\velhn,\ph^{n},\lstrhn)$, 
find $(\velh^{n+1},\ph^{n+1},\lstrh^{n+1}) \in (\Pt)^d\times\Pod\times(\Pz)^\dd$ such that,
 for any test function $(\bv,q,\bphi) \in (\Pt)^d\times\Pod\times(\Pz)^\dd$,
\begin{equation} \label{eq:log-P0-charac}
\begin{split}
0 = \intd & \Re\brk{\frac{\velhnp-\velhn}{\dt} 
+ \velhn\cdot\gvelhnp}\cdot\bv 
\\
& - \phnp\div\bv + q\div \velhnp+ (1-\e)\gvelhnp:\grad\bv
  +\frac{\e}{\Wi}\elstrhnp:\grad\bv
\\
&+ \brk{  \frac{\lstrhnp-\lstrhn \circ X^n(t^n) }{\dt} }:\bphi
- ( \Omh^{n+1}\lstrhnp - \lstrhnp\Omh^{n+1} ):\bphi
-2\B_h^{n+1}:\bphi
\\
&- \frac{1}{\Wi}(\emstrhnp-\I):\bphi,
\end{split}
\end{equation}
where the initial condition $(\velh^0,\ph^0,\lstrh^0) \in (\Pt)^d\times\Pod\times(\Pz)^\dd$ is given 
and where $X^n(t)$ is again defined by~\eqref{eq:flow-P0-charac}.

\paragraph{With the \emph{discontinuous Galerkin} method:} 
For a given $(\velhn,\ph^{n},\lstrhn)$, 
find $(\velh^{n+1},\ph^{n+1},\lstrh^{n+1}) \in (\Pt)^d\times\Pod\times(\Pz)^\dd$ such that,
for any test function $(\bv,q,\bphi) \in (\Pt)^d\times\Pod\times(\Pz)^\dd$,
\begin{equation} \label{eq:log-P0-DG}
\begin{split}
0 =  \sum_{k=1}^{N_K} \int_{K_k} &\Re\brk{\frac{\velhnp-\velhn}{\dt} 
+ \velhn\cdot\gvelhnp}\cdot\bv 
\\
& - \phnp\div\bv + q\div \velhnp+ (1-\e)\gvelhnp:\grad\bv
  +\frac{\e}{\Wi}e^{\lstrhnp}:\grad\bv
\\
&+ \brk{  \frac{\lstrhnp-\lstrhn}{\dt} }:\bphi
- ( \Omh^{n+1}\lstrhnp - \lstrhnp\Omh^{n+1} ):\bphi
-2\B_h^{n+1}:\bphi
\\
&- \frac{1}{\Wi}(\emstrhnp-\I):\bphi
\\
 + \sum_{j=1}^{N_E} &\int_{E_j} \Abs{\velhn\cdot\bn} \jump{\lstrhnp} :\bphi^+.
\end{split}
\end{equation}

\subsection{Local existence and uniqueness of the discrete solutions}

Before we show how to recover free energy estimates at the discrete level, let us now deal with 
the local-in-time existence and uniqueness of solutions to the discrete problems presented above.

First, since the mixed finite element space of Scott-Vogelius chosen in the systems above
for the velocity-pressure field satisfies the Babu\v{s}ka-Brezzi inf-sup condition, 
notice that the system~\eqref{eq:P0-charac} is equivalent to the following for all $n = 0 \dots N_T-1$:
For a given $(\velhn,\strhn)$, find $(\velhnp,\strhnp) \in (\Pt)^d_{\rm \div =0} \times (\Pz)^\dd$ such that,
for any test function $(\bv,\bphi) \in (\Pt)^d_{\rm \div =0} \times (\Pz)^\dd$,
\begin{equation} \label{eq:P0-charac-proj}
\begin{split}
0 = \intd & \Re\brk{\frac{\velhnp-\velhn}{\dt} 
+ \velhn\cdot\gvelhnp}\cdot\bv + (1-\e)\gvelhnp:\grad\bv
  +\frac{\e}{\Wi}\strhnp:\grad\bv
\\
&+ \brk{  \frac{\strhnp-\strhn \circ X^n(t^n) }{\dt} }:\bphi
- \brk{ (\gvelhnp)\strhnp +\strhnp(\gvelhnp)^T }:\bphi
\\
&+ \frac{1}{\Wi}(\strhnp-\I):\bphi,
\end{split}
\end{equation}
where the flow $X^n(t)$ is defined by~\eqref{eq:flow-P0-charac} and where we have used the following notation:
\begin{equation}\label{eq:P1-div=0}
(\Pt)^d_{\rm \div =0}=\left\{\bv \in (\Pt)^d, \,  \intd q \div \bv = 0, \ \forall q \in \Pod \right\}.
\end{equation}

Notice that it is also straightforward to rewrite the systems~\eqref{eq:P0-DG},~\eqref{eq:log-P0-charac}
and~\eqref{eq:log-P0-DG} using $\velh\in(\Pt)^d_{\rm \div =0}$ instead of $(\velh,\ph)\in(\Pt)^d\times\Pod$.
For instance, the system~\eqref{eq:log-P0-charac} is equivalent to: 
For a given $(\velhn,\lstrhn)$, find $(\velhnp,\lstrhnp)\in (\Pt)^d_{\div=0}\times(\Pz)^\dd$ such that,
for all $(\bv,\bphi) \in (\Pt)^d_{\div=0}\times(\Pz)^\dd$,
\begin{equation} \label{eq:log-P0-charac-proj}
\begin{split}
0 = \intd & \Re\brk{\frac{\velhnp-\velhn}{\dt} + \velhn\cdot\gvelhnp}\cdot\bv 
+ (1-\e)\gvelhnp:\grad\bv +\frac{\e}{\Wi}\elstrhnp:\grad\bv
\\
& + \brk{  \frac{\lstrhnp-\lstrhn \circ X^n(t^n) }{\dt} }:\bphi
- ( \Omh^{n+1}\lstrhnp - \lstrhnp\Omh^{n+1} ):\bphi
-2\B_h^{n+1}:\bphi
\\
& - \frac{1}{\Wi}(\emstrhnp-\I):\bphi.
\end{split}
\end{equation}

Then, we have the:
\begin{proposition} \label{prop:existence}
For any couple $(\velhn,\strhn)$ with $\strhn$ symmetric positive definite, 
there exists $c_0 \equiv c_0\brk{\velhn,\strhn} > 0$ such that, for all $0 \le \dt < c_0$, 
there exists a unique solution $(\velh^{n+1},\strh^{n+1})$ to the system~\eqref{eq:P0-charac} (resp.~\eqref{eq:P0-DG})
with $\strh^{n+1}$ symmetric positive definite.
\end{proposition}

\begin{proof}[Proof of Proposition \ref{prop:existence}]
The proofs for systems~\eqref{eq:P0-charac} and~\eqref{eq:P0-DG} are fully similar,
so we will proceed with the proof for system~\eqref{eq:P0-charac} only,
using its rewriting as system~\eqref{eq:P0-charac-proj}.

For a given mesh ${\cal T}_h$, let us denote by $Y^{n+1} \in \R^{2N_D+3N_K}$ the vector whose entries are respectively
the nodal and elementwise values of $(\velh^{n+1},\strh^{n+1})$, solution to the system~\eqref{eq:P0-charac-proj}.
The system of equations~\eqref{eq:P0-charac-proj} rewrites in terms of the vector $Y^{n+1} \in \R^{2N_D+3N_K}$ as:
for a given $Y^n$ and $\dt$, find a zero $Y^{n+1}$ of the application $Q$ defined by
\begin{equation}\label{eq:Q}
Q(\dt,Y^{n+1}) = \dt A(Y^{n+1})Y^{n+1} + \dt B(Y^n)Y^{n+1} + Y^{n+1} -
C(Y^n,\dt) \,,
\end{equation}
where $A$ and $B$ are linear continuous matrix-valued functions in $\R^{(2N_D+3N_K) \times (2N_D+3N_K)}$,
and where $C$ is a vector-valued function in $\R^{2N_D+3N_K}$ 
(notice that the dependence of the function $C$ on $\dt$ is only related to
the computation of the backward flow during a time step $\dt$,
so that $C(Y^n,0)=Y^n$, and with the DG method it simplifies as $C(Y^n,\dt)=Y^n$).
The functions $A$, $B$ and $C$ also implicitly depend on ${\cal T}_h$,
as well as on the parameters $\Re,\Wi,\e$.

Now, $Q(\dt,Y)$ is continuously differentiable with respect to $(\dt,Y)$
and we have, with $I$ the identity matrix in $\R^{(2N_D+3N_K) \times
  (2N_D+3N_K)}$:
\begin{equation}\label{eq:grad_Q}
\nabla_Y Q(\dt,Y) = I + \dt\big( B(Y^n) + A(Y) + (\nabla_Y A)Y \big).
\end{equation}
Then, for given vectors $Y^n$ and $Y$, the matrix $ \nabla_Y Q(\dt,Y) $ is invertible
for all $0 \le \dt \le \Norm{ B(Y^n) + A(Y) + (\nabla_Y A)Y }^{-1}$
(with convention $\Norm{ B(Y^n) + A(Y) + (\nabla_Y A)Y }^{-1}=\infty$ if $B(Y^n) + A(Y) + (\nabla_Y A)Y = 0$),
and then defines an isomorphism in $\R^{2N_D+3N_K}$.

Let us denote by $S_+^*$ the subset of $\R^{2N_D+3N_K}$ that only contains vectors corresponding to
elementwise values of \emph{positive definite} matrix-valued functions $\strh$ in $\D$.
Since $\mathcal{S}_+^*(\R^{d \times d})$ is an open (convex) domain of $\R^{d \times d}$,
$S_+^*$ is clearly an open (convex) domain of $\R^{2N_D+3N_K}$.

Since $Q(0,Y^n)=0$ and $\nabla_Y Q(0,Y^n)$ is invertible, 
in virtue of the implicit function theorem, there exist a neighborhood $[0,c_0)\times V(Y^n)$ of $(0,Y^n)$ in $\R_+ \cap S_+^*$ and a continuously differentiable function $R:[0,c_0)\rightarrow V(Y^n)$,
such that, for all $0\le\dt<c_0$:
$$ Y = R(\dt) \Longleftrightarrow Q(\dt,Y) = 0. $$
For a given time step $\dt\in[0,c_0)$ and a given symmetric positive definite tensor field $\strhn$, 
$ R(\dt) \in V(Y^n)$ is the vector of values 
for a symmetric positive definite solution $\strhnp$ to the system~\eqref{eq:P0-charac-proj}.
Notice that, up to this point, $c_0 = c_0(Y^n)$ is function of $Y^n$, as well as $\Re,\Wi,\e$ and ${\cal T}_h$.
\end{proof}

For solutions $(\velhn,\strhn)$ to the systems~\eqref{eq:log-P0-charac}
and~\eqref{eq:log-P0-DG}, we similarly have:
\begin{proposition} \label{prop:log-existence}
For any couple $(\velhn,\lstrhn)$, 
there exists a constant $c_0 \equiv c_0(\velhn,\lstrhn) > 0$ such that, for all $0 \le \dt < c_0$, 
there exists a unique solution $(\velhnp,\lstrhnp)$ to the system~\eqref{eq:log-P0-charac} (resp.~\eqref{eq:log-P0-DG}).
\end{proposition}

The proof of Proposition~\ref{prop:log-existence} is fully similar to 
that of the Proposition~\ref{prop:existence},
but for the expressions of $Q(\dt,Y)$ with respect to $Y$.
An additional term $\dt D(Y)$ now appears in $Q$ due to $\elstrhnp$.
This term is continuously differentiable with respect to $Y$,
and the derivative $\nabla_Y Q(0,Y^n)$ is still invertible.
Thus, the proof can be completed using similar arguments.

Anticipating the results of Section~\ref{sec:stability}, we would like to mention
that the above results will be extended in two directions, using the
discrete free energy estimates which will be proved in the following.
\begin{itemize}
\item We will show that the constant $c_0$ in
  Proposition~\ref{prop:existence} (resp. Proposition~\ref{prop:log-existence}) can be
  chosen independently of $(\velhn,\strhn)$ (resp. $(\velhn,\lstrhn)$), which yields a longtime
  existence and uniqueness result for the solutions to the discrete problems (see
  Propositions~\ref{prop:stability} and~\ref{prop:log-stability} below).
\item We will also show, but for the log-formulation only,
  that it is possible to prove a long-time existence
  result without any restriction on the time step $\dt$
  (see Proposition~\ref{prop:log-stability-2} below).
\end{itemize}

\section{Discrete free energy estimates with piecewise constant discretization of the stress fields $\strh$ and $\lstrh$}
\label{sec:P0}

In this section, 
we prove that various numerical schemes with piecewise constant $\strh$ or $\lstrh$ 
satisfy a discrete free energy estimate. 
We first concentrate on Scott-Vogelius finite element spaces for $(\velh,\ph)$ 
and then address the case of other mixed finite element spaces in Section~\ref{sec:need}.

\subsection{Free energy estimates with piecewise constant discretization of $\strh$}

\subsubsection{The characteristic method}
\label{sec:P0-strh-charac}

\begin{proposition} \label{prop:P0-charac}
Let $(\velhn,\phn,\strhn)_{0 \le n \le N_T}$ be solution to~\eqref{eq:P0-charac}, 
such that $\strh^n$ is positive definite.
Then, the free energy of the solution $(\velhn,\phn,\strhn)$:
\begin{equation} \label{eq:Fhn}
F_h^n=F(\velh^n,\strh^n)  = \frac{\Re}{2}\intd|\velh^n|^2 + \frac{\e}{2 \Wi}\intd\tr(\strh^n-\ln\strh^n-\I)\,,
\end{equation}
satisfies:
\begin{equation} \label{eq:free-energy-P0-charac}
\begin{split}
F_h^{n+1} - F_h^n & + \intd \frac{\Re}{2} |\velhnp-\velhn|^2 \\
& + \dt \intd (1-\e)|\gvelhnp|^2 + \frac{\e}{2 {\Wi}^2} \tr\brk{\strhnp+(\strhnp)^{-1}-2I} \le 0.
\end{split}
\end{equation}
In particular, the sequence $(F_h^n)_{0 \le n \le N_T}$ is non-increasing.
\end{proposition}

\begin{proof}[Proof of Proposition \ref{prop:P0-charac}]
Let $(\velhnp,\phnp,\strhnp)$ be a solution to system~\eqref{eq:P0-charac}. 
 Notice that $(\strhnp)^{-1}$ is still in $(\Pz)^\dd$.
We can thus use $(\velhnp,\phnp,\frac{\e}{2\Wi}\brk{\I-(\strhnp)^{-1})}$ as a test function 
in the system~\eqref{eq:P0-charac}, which yields:
\begin{equation}
\begin{split}
0 = \intd & \Bigg[\Re\brk{\frac{|\velhnp|^2-|\velhn|^2}{2\dt} +
\frac{|\velhnp-\velhn|^2}{2\dt} 
+ \velhn\cdot\grad\frac{|\velhnp|^2}{2}} 
\\
& - \phnp\div\velhnp + \phnp\div \velhnp+ (1-\e)|\gvelhnp|^2
  +\frac{\e}{\Wi}\strhnp:\gvelhnp\Bigg]
\\
& + \frac{\e}{2\Wi} \Bigg[\brk{  \frac{\strhnp-\strhn \circ X^n(t^n)
}{\dt} }:(\I-(\strhnp)^{-1})
\\
& - 2(\gvelhnp)\strhnp:(\I-(\strhnp)^{-1})+
\frac{1}{\Wi}(\strhnp-\I):(\I-(\strhnp)^{-1})\Bigg].
\end{split}
\end{equation}

We first examine the terms associated with momentum conservation and incompressibility.
We recall that $\velhnp$ satisfies~\eqref{eq:divu=0} since we use Scott-Vogelius finite elements.
By the Stokes theorem (using the no-slip boundary condition), we immediately obtain:
\[
 \intd \velhn\cdot\grad{|\velhnp|^2}=  - \intd (\div\velhn) {|\velhnp|^2} = 0.
\]
The terms involving $\phnp$ also directly cancel.
We now consider the terms involving $\strhnp$. 
The upper-convective term in the tensor derivative rewrites:
\[
(\gvelhnp)\strhnp:(\I-(\strhnp)^{-1}) = \strhnp:\gvelhnp - \div \velhnp,
\]
which vanishes after combination with the extra-stress term $\strhnp:\gvelhnp$
in the momentum conservation equation, 
and using the incompressibility property.
The last term rewrites:
\[
 (\strhnp-\I):(\I-(\strhnp)^{-1}) = \tr\brk{\strhnp+(\strhnp)^{-1}-2\I}. 
\]
The remaining term writes:
\begin{align*}
\intd \brk{\strhnp-\strhn \circ X^n(t^n)}:(\I-(\strhnp)^{-1}) 
= \intd &\tr(\strhnp) - \tr(\strhn \circ X^n(t^n)) \\
 &+ \tr\left( [\strhn \circ X^n(t^n)] [\strhnp]^{-1} -\I \right).
\end{align*}
We first make use of~\eqref{eq:traceAB} with $\strs=\strhn \circ X^n(t^n)$ and
$\str=\strhnp$:
\begin{equation*}
\tr\left( [\strhn \circ X^n(t^n)] [\strhnp]^{-1} -\I \right)
 \geq   \tr\ln(\strhn \circ X^n(t^n)) - \tr\ln(\strhnp).
\end{equation*}
Then, we have:
\[
\intd - \tr\brk{\strhn \circ X^n(t^n) + \ln(\strhn \circ X^n(t^n))} = \intd - \tr\brk{\strhn + \ln\strhn}\,,
\]
since the strong incompressibility property ($\div \velhn = 0$) implies 
that the flow $X^n(t)$ defines a mapping 
with constant Jacobian equal to $1$ for all $t \in [t^n,t^{n+1}]$.
Finally, we get the following lower bound:
\[
\intd \brk{\strhnp-\strhn \circ X^n(t^n)}:(\I-(\strhnp)^{-1})
\geq 
\intd \tr(\strhnp - \ln \strhnp) - \tr(\strhn - \ln \strhn),
\]
hence the result~\eqref{eq:free-energy-P0-charac}.

Notice that $\tr\brk{\strhnp+(\strhnp)^{-1}-2\I} \ge 0$ in virtue of 
the equation~\eqref{eq:tr-plus_inv}, which shows that 
the sequence $(F_h^n)_{0 \le n \le N_T}$ is non-increasing.
\end{proof}

\subsubsection{The discontinuous Galerkin method}
\label{sec:P0-strh-DG}

\begin{proposition} \label{prop:P0-DG}
Let $(\velhn,\phn,\strhn)_{0 \le n \le N_T}$ be solution to~\eqref{eq:P0-DG}, such that $\strhn$ is positive definite. Then, the free energy $F_h^n$ defined by~\eqref{eq:Fhn} satisfies the free energy estimate~\eqref{eq:free-energy-P0-charac}. In particular, the sequence $(F_h^n)_{0 \le n \le N_T}$ is non-increasing.
\end{proposition}

\begin{proof}[Proof of Proposition \ref{prop:P0-DG}]
We only point out the differences with the proof of Proposition~\ref{prop:P0-charac}.
They consist in the treatment of the discretization of the advection terms for $\strh$. 
We recall that the test function in stress is $\bphi = \frac{\e}{2 \Wi}(\I-(\strhnp)^{-1})$, so that we have:
\begin{align*}
\sum_{j=1}^{N_E} &\int_{E_j} \Abs{\velhn\cdot\bn} \jump{\strhnp}:\brk{\I-(\strhnp)^{-1}}^+ \\
&=
\sum_{j=1}^{N_E} \int_{E_j} \Abs{\velhn\cdot\bn} \jump{\tr(\strhnp)} 
+
\sum_{j=1}^{N_E} \int_{E_j} \Abs{\velhn\cdot\bn} \tr \brk{\strh^{n+1,-}(\strh^{n+1,+})^{-1}-\I}.
\end{align*}
Again, we make use of~\eqref{eq:traceAB}, with $\strs = \strh^{n+1,-}$ and $\str = \strh^{n+1,+}$:
\[
\tr\brk{\strh^{n+1,-}(\strh^{n+1,+})^{-1}-\I}\geq\tr\brk{\ln\strh^{n+1,-}-\ln\strh^{n+1,+}}.
\]
We get, by Formula~\eqref{eq:jump}, the fact that $\strh^{n+1} \in (\Pz)^\dd$, the Stokes theorem 
and the incompressibility property~\eqref{eq:divu=0}:
\begin{align}
\sum_{j=1}^{N_E} \int_{E_j} \Abs{\velhn\cdot\bn} \jump{\strhnp} : \left(\I-(\strhnp)^{-1}\right)^+ &\geq 
\sum_{j=1}^{N_E} \int_{E_j} \Abs{\velhn\cdot\bn} \jump{\tr( \strh^{n+1} - \ln \strh^{n+1} )},\nonumber\\
& =
- \sum_{k=1}^{N_K} 
\int_{\partial K_k}\brk{\velhn\cdot \bn_{K_k}}\tr( \strh^{n+1} - \ln \strh^{n+1} ), \nonumber\\
& =
- \sum_{k=1}^{N_K} 
\left( \tr( \strh^{n+1} - \ln \strh^{n+1} ) \right)\Big|_{K_k} \int_{\partial K_k} \velhn\cdot \bn_{K_k},  \nonumber \\
& =
- \sum_{k=1}^{N_K} 
\left( \tr( \strh^{n+1} - \ln \strh^{n+1} ) \right)\Big|_{K_k} \int_{K_k} \div(\velhn),\nonumber \\ 
& = 0. \label{eq:preuve_DG}
\end{align}
Moreover, it is easy to prove the following, 
using the same technique as in the proof of Proposition~\ref{prop:P0-charac}:
\[
\intd \brk{\strhnp-\strhn}:(\I-(\strhnp)^{-1})
\geq 
\intd \tr(\strhnp - \ln \strhnp) - \tr(\strhn - \ln \strhn).
\]
This concludes the proof.
\end{proof}

\subsection{Free energy estimates with piecewise constant discretization of $\lstrh$}

This section is the equivalent of the previous section for the log-formulation.

\subsubsection{The characteristic method}
\label{sec:log-P0-charac}

\begin{proposition} \label{prop:log-P0-charac}
Let $(\velhn,\phn,\lstrhn)_{0 \le n \le N_T}$ be solution to~\eqref{eq:log-P0-charac}.
Then, the free energy of the solution $(\velhn,\phn,\lstrhn)$:
\begin{equation} \label{eq:log-Fhn}
F_h^n=F(\velhn,\elstrhn)  = \frac{\Re}{2}\intd|\velhn|^2 + \frac{\e}{2 \Wi}\intd\tr(\elstrhn-\lstrhn-\I)\,,
\end{equation}
satisfies:
\begin{equation} \label{eq:free-energy-log-P0-charac}
\begin{split}
F_h^{n+1} - F_h^n & + \intd \frac{\Re}{2} |\velhnp-\velhn|^2 \\
& + \dt \intd (1-\e)|\gvelhnp|^2 + \frac{\e}{2 {\Wi}^2} \tr\brk{\elstrhnp+\emstrhnp-2I} \le 0.
\end{split}
\end{equation}
In particular, the sequence $(F_h^n)_{0 \le n \le N_T}$ is non-increasing.
\end{proposition}

\begin{proof}[Proof of Proposition \ref{prop:log-P0-charac}]
We shall use as test functions $\brk{\velhnp,\phnp,\frac{\e}{2\Wi}(e^{\lstrhnp} -\I)}$ in~\eqref{eq:log-P0-charac}. 
We emphasize that,
as long as the solution $(\velh^{n+1},\ph^{n+1},\lstrh^{n+1})$ exists (see Proposition~\ref{prop:log-existence}), 
$e^{\lstrhnp}$ is well-defined, symmetric positive definite and piecewise constant. 

The terms are treated similarly as in the proof of Proposition \ref{prop:P0-charac}. 
For the material derivative of $\lstrh$, we have:
\begin{align*}
 \intd \brk{\lstrhnp-\lstrhn\circ X^n(t^n)}:(e^{\lstrhnp} - \I)
 = & \intd  \brk{\lstrhnp-\lstrhn\circ X^n(t^n)}:e^{\lstrhnp} \\
   & \quad - \tr \brk{\lstrhnp-\lstrhn\circ X^n(t^n)} .
\end{align*}
Using the equation~\eqref{eq:traceAB+} with $\strs=e^{\lstrhnp}$ and $\str=e^{\lstrhn \circ X^n(t^n)}$, we obtain:
\[
\brk{\lstrhnp-\lstrhn\circ X^n(t^n)}:e^{\lstrhnp} \ge \tr(e^{\lstrhnp} - e^{\lstrhn\circ X^n(t^n)}),
\]
and thus:
\begin{align*}
 \intd \brk{\lstrhnp-\lstrhn\circ X^n(t^n)}:(e^{\lstrhnp} - \I)
 & \ge \intd \tr (e^{\lstrhnp} - \lstrhnp) - \intd \tr (e^{\lstrhn} - \lstrhn)\circ X^n(t^n),\\
 & = \intd \tr (e^{\lstrhnp} - \lstrhnp) - \intd \tr (e^{\lstrhn} - \lstrhn),
\end{align*}
where the fact that the Jacobian of the flow $X^n$ is constant equal to one
(because $\velhn$ is divergence free) has been used in the change of variable in the last equality.

Besides, using the equation~\eqref{eq:omegapsi0}, we have:
\begin{align*}
\intd ( \Omh^{n+1}\lstrhnp - \lstrhnp\Omh^{n+1} ):(e^{\lstrhnp} -\I)
&=\intd ( \Omh^{n+1}\lstrhnp - \lstrhnp\Omh^{n+1} ):e^{\lstrhnp},\\ 
&=0.
\end{align*}

Last, using~\eqref{eq:uB}:
\begin{align*}
\intd \B_h^{n+1}:(e^{\lstrhnp} -\I)
&=\intd \B_h^{n+1}:e^{\lstrhnp} - \intd \tr(\B_h^{n+1}),\\
&=\intd \gvelhnp:e^{\lstrhnp} - \intd \div(\velhnp),\\
&=\intd \gvelhnp:e^{\lstrhnp},
\end{align*}
which cancels out with the same term $\intd e^{\lstrhnp}:\gvelhnp$ in the momentum equation.
\end{proof}

\subsubsection{The discontinuous Galerkin method}
\label{sec:log-P0-DG}

\begin{proposition} \label{prop:log-P0-DG}
Let $(\velhn,\phn,\lstrhn)_{0 \le n \le N_T}$ be solution to~\eqref{eq:log-P0-DG}. Then, the free energy $F_h^n$ defined by~\eqref{eq:log-Fhn} satisfies the free energy estimate~\eqref{eq:free-energy-log-P0-charac}. In particular, the sequence $(F_h^n)_{0 \le n \le N_T}$ is non-increasing.
\end{proposition}

The proof is straightforward using elements of the proofs of Proposition~\ref{prop:log-P0-charac} and Proposition~\ref{prop:P0-DG}. 

\subsection{Other finite elements for $(\velh,\ph)$}
\label{sec:need}

In this section, we review some finite element spaces for $(\velh,\ph)$ other than Scott-Vogelius
for which the results of the last two sections still hold. 

First, let us stress the key arguments we used in the proofs above. 
If the advection terms $\vel \cdot \grad \strs$ and $\vel \cdot \grad \lstr$ are discretized by 
the \emph{characteristic} method, we need the velocity field $\velhn$ to be divergence free:
\begin{equation}\label{eq:strong-incompressibility}
\div \velhn = 0,
\end{equation}
in order for the flow $X^n$ satisfying~\eqref{eq:flow-P0-charac} to be with Jacobian one.
When $\velhn$ is only piecewise smooth (consider below the case of $\Pod$ velocity fields),
the divergence in the left-hand side of~\eqref{eq:strong-incompressibility} should be understood in the distribution sense. 
By the way, the equation~\eqref{eq:strong-incompressibility} univoquely defines 
the trace of the normal component $\velhn\cdot\bn$ on the edges of the mesh,
which is a sufficient condition to define the flow associated with the vector field $\velhn$ through \eqref{eq:flow-P0-charac},
and which is necessary to treat the advection term in the Navier-Stokes equation.

If the advection terms are discretized by the \emph{discontinuous Galerkin} method, it is necessary that the trace of the normal component of $\velh$  be univoquely defined on the edges of the mesh since it appears in the jump terms  $\sum_{j=1}^{N_E} \int_{E_j}  \Abs{\velhn\cdot\bn} \jump{\strhnp} : \bphi^+$ or $\sum_{j=1}^{N_E} \int_{E_j}  \Abs{\velhn\cdot\bn}  \jump{\lstrhnp} : \bphi^+$ in the variational formulations. But to obtain~\eqref{eq:preuve_DG}, and contrary to the characteristic method, only the following \emph{weak} incompressibility property is needed:
\begin{equation*}
\forall k=1\dots N_K, \int_{K_k} \div \velhn = 0,
\end{equation*}
which is equivalent to write
\begin{equation}\label{eq:weak-incompressibility}
\forall q \in\mathbb{P}_0, \intd \div (\velhn) q = 0.
\end{equation}

These properties needed to obtain the discrete free energy estimates are summarized in Table~\ref{tab:P0-summary}.

\begin{table}[!h]
\begin{center}
\begin{tabular}[t]{|l||p{5cm}|p{4cm}|}
\hline
Advection discretized by: &Characteristics  & DG \\ 
\hline
Requirements for $\velh$:
& 
$\div \velh = 0$  \newline
( $\Rightarrow \det(\nabla_{\x}X^n) \equiv 1 $ ) \newline
( $\Rightarrow \brk{\velh \cdot \bn}|_{E_j} $ well defined )
&
$ \intd q \div \velh = 0$, $\forall q \in \Pz$ \newline
\ and \newline
$ \brk{\velh \cdot \bn}|_{E_j} $ well defined \\
\hline
\end{tabular}
\caption{\label{tab:P0-summary} Summary of the arguments  
with $(\velh,\ph,\strh)$ or $(\velh,\ph,\lstrh)$ in $(\Pt)^d\times\Pod\times(\Pz)^\dd$}
\end{center}
\end{table}

Below, we consider the following alternative choices of the finite elements space for $(\velh,\ph)$:
\begin{itemize}
\item the Taylor-Hood finite element space: $(\velh,\ph) \in (\Pt)^d\times\Po$, which satisfies the Babu\v{s}ka-Brezzi inf-sup condition, whatever the mesh ;
\item the mixed Crouzeix-Raviart finite element space (see~\cite{crouzeix-raviart-73}): $(\velh,\ph) \in (\Po^{CR})^d \times \Pz$, where
\begin{equation}\label{eq:P1CR}
 \mathbb{P}_{1}^{CR}=\left\{v \in \Pod, \forall E_j, \int_{E_j}\jump{v}=0 \right\},
\end{equation}
which also satisfies the Babu\v{s}ka-Brezzi inf-sup condition, whatever the mesh ;
\item stabilized formulations for $(\velh,\ph) \in (\Po)^d\times\Po$ 
or $(\velh,\ph) \in (\Po)^d\times\Pz$.
\end{itemize}
This is not exhaustive, but it is sufficient to highlight which
modifications are needed in the variational formulations, compared to
the Scott-Vogelius mixed finite element, for the discrete free energy
estimates to hold. In particular, some projection of the velocity field
is needed in the discretization of the advection terms $\vel \cdot \grad
\strs$ and $\vel \cdot \grad \lstr$ in order to satisfy the requirements
of Table~\ref{tab:P0-summary}. These projection operators are introduced
in the next Section~\ref{sec:projections}.

The results of Section~\ref{sec:need} are summarized in Table~\ref{tab:finite element-summary}.

\subsubsection{Some useful projection operators for the velocity field}
\label{sec:projections}

Let us introduce three projection operators for the velocity field.

Following \cite{pironneau-82}, we first define the orthogonal projection $\Phrot$ of $(\Pod)^d$ 
onto the piecewise constant solenoidal vector fields
built from affine continuous scalar fields:
$$ \{ \nabla \times \zeta_h | \zeta_h \in (\Po)^d, \zeta_h \times \bn =
0 \text{ on } \partial\D \}. $$
We suppose here that $d=3$, but the extension to the case $d=2$ is straightforward.
We set $\Phrot\brk{\velh} = \nabla \times \psi_h$ where 
$\psi_h \in (\Po)^d$, such that $\psi_h\times\bn|_{\partial\D}=0$, satisfies :
$$ \intd (\nabla \psi_h) : (\nabla \zeta_h) = \intd \velh \cdot (\nabla \times \zeta_h),
\forall \zeta_h \in (\Po)^d, \zeta_h \times\bn|_{\partial\D}=0. $$
Because $\Phrot\brk{\velh}$ is solenoidal, we always have the strong incompressibility property~\eqref{eq:strong-incompressibility}:
$$\div \Phrot\brk{\velh} = 0\,,$$ 
for any velocity field $\velh$. Of course, this operator is consistent
only for divergence free velocity fields $\velh$ (or velocity field
$\velh$ with
vanishing divergence when $h$ goes to zero). See~\cite{pironneau-82} for
consistency results.

Second, following \cite{ern-guermond-04}, we define the Raviart-Thomas interpolator $\PhRTO$ 
from $(\Pod)^d$ 
onto the vector subspace of $(\Pod)^d$ made of the vector fields in
$(\Pz)^d + \x \Pz$ 
with continuous normal component across the edges $E_j$
(whose trace on $E_j$ is then univoquely defined).
The projection $\PhRTO\brk{\velhn}$ clearly satisfies, for any element $K_k$:
\begin{equation}\label{eq:div-conservation}
\begin{split}
\int_{K_k} \div \velh & = \int_{\partial K_k} \velhn \cdot \bn_{K_k}, 
\\
& =  \int_{\partial K_k} \PhRTO\brk{\velhn} \cdot \bn_{K_k},
\\
& = \int_{K_k} \div \PhRTO\brk{\velhn}.
\end{split}
\end{equation}
Thus, it satisfies the weak incompressibility property \eqref{eq:weak-incompressibility}:
$$\forall q \in\mathbb{P}_0, \intd \div (\PhRTO\brk{\velhn}) q = 0\,,$$
if, and only if, the velocity field $\velhn$ also satisfies it.

Third, we define $\PhBDM$ as the Brezzi-Douglas-Marini projection operator \cite{brezzi-douglas-marini-85,brezzi-douglas-marini-86}. It is with value in $(\Po)^d$.
This projection operator satisfies the same divergence preservation
property \eqref{eq:div-conservation} than $\PhRTO$, but is of better accuracy. 

Note that $\PhBDM$ like $\PhRTO$ are local interpolating operators in the sense that all the computations can be made elementwise. This is not the case for $\Phrot$.

In addition, we will need the following lemma:
\begin{lemma}\label{lem:proj_div_0}
For any velocity field $\velhn$ such that the previously defined
interpolating operators are well defined, the normal components of the interpolated vector field, 
$\Phrot\brk{\velhn}\cdot\bn$, $\PhRTO\brk{\velhn}\cdot\bn$ and $\PhBDM\brk{\velhn}\cdot\bn$ are also well defined 
on any internal edges $E_j$.

Moreoever, if $\velhn \in (\Pod)^d$ is a velocity field such that, for all $k=1 \ldots N_K$:
$$ \int_{K_k}\div(\velhn)=0\,,$$ 
then $\div(\PhRTO\brk{\velhn})=\div(\PhBDM\brk{\velhn})=0$ (in the sense of distribution). 
\end{lemma}
\begin{proof}
By construction, $\PhRTO$ and $\PhBDM$ are with value in velocity fields such that their normal component is 
continuous across the edges. 
This is also the case for $\Phrot$ since $\Phrot$ is with value in divergence free velocity fields.
Then, from the equation~\eqref{eq:div-conservation}, we have $\int_{K_k}\div(\PhRTO\brk{\velhn})=0$. 
Since $\div(\PhRTO\brk{\velhn})$ is in $(\Pz)^d$, this shows that $\div(\PhRTO\brk{\velhn})$ is zero in any element $K_k$. 
Finally, $\PhRTO\brk{\velhn}$ has continuous normal components across the edges of the mesh. 
This shows that $\div(\PhRTO\brk{\velhn})=0$ in the sense of distribution. 
The same proof holds for the projection operator $\PhBDM$.
\end{proof}

\subsubsection{Alternative mixed finite element space for $(\velh,\ph)$ with inf-sup condition}
\label{sec:P0-infsup}

In this section, we show how to derive discrete free energy estimates with mixed finite element spaces 
for the velocity and pressure fields which satisfy the inf-sup condition, 
but which are not the Scott-Vogelius finite elements.

Let us first consider the \emph{Taylor-Hood} element for $(\velh,\ph)$, that is $(\Pt)^d\times\Po$. In this case, since the velocity field $\velh$ is not divergence free either in the weak form \eqref{eq:weak-incompressibility}, or in the strong form \eqref{eq:strong-incompressibility}, a projection of the velocity field is required in the discretization of the advection terms $\vel \cdot \grad \strs$ and $\vel \cdot \grad \lstr$. More precisely, we need to use the projection velocity $\Phrot{\velh^n}$ (and, among the three projection operators we introduced above, this is the only one which is such that the strong or weak incompressibility is satisfied). 
For the \emph{characteristic} method, one uses the flow $X^n(t)$ satisfying:
\begin{equation} \label{eq:flow-P0-charac-HT}
\left\{
\begin{array}{l}
 \deriv{}{t} X^n(t,x) = \Phrot{\velh^n} ( X^n(t,x) ), \quad \forall t \in [t^n,t^{n+1}],\\
 X^n(t^{n+1},x) = x. 
\end{array}
\right.
\end{equation}
For the \emph{discontinuous Galerkin} method, the advection term in the conformation-tensor formulations writes
 (see the last line in~\eqref{eq:P0-DG}):
$$
+ \sum_{j=1}^{N_E} \int_{E_j} |\Phrot\brk{\velhn}\cdot\bn| \jump{\strhnp} :\bphi^+.
$$
Notice that in the terms $\jump{\strhnp} :\bphi^+$, 
the projected velocity $\Phrot{\velh^n}$ is used to define the upstream and downstream values following~\eqref{eq:sign-convention}. 
Another modification, which is specific to the Navier-Stokes equation, is needed to treat the advection term on the velocity.
Namely, one needs to add to the weak formulation the so-called Temam correction term (see~\cite{temam-66}):
\begin{equation}\label{eq:temam-term} 
 + \frac{\Re}{2} \intd \div\brk{\velhn}(\bv\cdot\velhnp) 
\end{equation}
in such a way that, when $\velhnp$ is used as a test function:
$$
\Re \intd \brk{ \velhn  \cdot \gvelhnp }\cdot \velhnp + \frac{\Re}{2} \intd \div\brk{\velhn} |\velhnp|^2 = 0.
$$
With these modifications (projection of the velocity field in the advection terms, and Temam correction terms), the free energy estimate~\eqref{eq:free-energy-P0-charac} is satisfied by the scheme. Similar results (discrete free energy estimates for $(\velh,\ph,\lstrh)$ in $(\Pt)^d\times\Po\times(\Pz)^\dd$) can be proved on the log-formulation.

Let us now discuss the use of \emph{Crouzeix-Raviart} finite elements for velocity: 
$(\velh,\ph,\strh)$ in $(\Po^{CR})^d\times\mathbb{P}_0\times(\Pz)^\dd$ (see~\eqref{eq:P1CR}). 
In this case, the Navier-Stokes equations can be discretized using a characteristic method:
\begin{equation} \label{eq:P0-charac-CR}
\begin{split}
 0 = \sum_{k=1}^{N_K} \int_{K_k} & \Re\brk{\frac{\velhnp-\velhn \circ X^n(t_n)}{\dt}}\cdot\bv\\ 
 &  - \phnp\div\bv + q\div \velhnp+ (1-\e)\gvelhnp:\grad\bv +\frac{\e}{\Wi}\strhnp:\grad\bv,
\end{split}
\end{equation}
where $X^n$ is obtained from the projected velocity field $\Ph{\velh^n}$ as:
\begin{equation} \label{eq:flow-P0-charac-CR}
\left\{
\begin{array}{l}
 \deriv{}{t} X^n(t) = \Ph{\velh^n} ( X^n(t) ), \quad \forall t \in [t^n,t^{n+1}],\\
 X^n(t^{n+1}) = x.
\end{array}
\right.
\end{equation}
The projected velocity $\Ph{\velh^n}$ is defined using any of the three projectors presented above, 
that is $\Phrot{\velh^n}$, $\PhRTO{\velh^n}$ or $\PhBDM{\velh^n}$.  
The Navier-Stokes equations can also be discretized using a \emph{discontinuous Galerkin} formulation:
\begin{equation} \label{eq:P0-DG-CR}
\begin{split}
0 = & \sum_{k=1}^{N_K} \int_{K_k} \Re\brk{\frac{\velhnp-\velhn}{\dt} + \Ph\brk{\velhn}\cdot\gvelhnp}\cdot\bv 
 + \Re \sum_{j=1}^{N_E} \int_{E_j} \Abs{\Ph\brk{\velhn}\cdot\bn} \jump{\velhnp}\cdot\BRK{\bv}\\
& + \sum_{k=1}^{N_K} \int_{K_k} 
- \phnp\div\bv + q\div \velhnp+ (1-\e)\gvelhnp:\grad\bv
  +\frac{\e}{\Wi}\strhnp:\grad\bv.
\end{split}
\end{equation}
Here again, $\Ph{\velh^n}$ is any of the three projectors presented above. We would like to mention that we are not aware that the projector $\Phrot$ has ever been used with discontinuous Galerkin methods,
so that the consistency of the discontinuous Galerkin approach combined
with this projector still needs to be investigated.

Likewise, the advection term $\vel \cdot \grad \strs$ in the equation on the stress can be treated by the characteristic method or the discontinuous Galerkin method,
as above for the advection term in the Navier-Stokes equations.

Notice that whatever the projecting operator used, $\div(\Ph{\velh^n})=0$ holds (see Lemma~\ref{lem:proj_div_0} above). With this property, it is easy to check that Propositions~\ref{prop:P0-charac} and~\ref{prop:P0-DG} still hold for this finite element. For example, the advection term in the Navier-Stokes equations is treated as follows (using the fact that $\div(\Ph\brk{\velhn})=0$ and~\eqref{eq:jump}):
\[
\begin{split}
\sum_{k=1}^{N_K} \int_{K_k} & \brk{\Ph\brk{\velhn}\cdot\gvelhnp}\cdot\velhnp + \sum_{j=1}^{N_E} \int_{E_j} \Abs{\Ph\brk{\velhn}\cdot\bn} \jump{\velhnp}\cdot\BRK{\velhnp}
\\
& =
\sum_{k=1}^{N_K} \int_{K_k} \div\brk{ \Ph\brk{\velhn} \frac{|\velhnp|^2}{2} } + \sum_{j=1}^{N_E} \int_{E_j} \Abs{\Ph\brk{\velhn}\cdot\bn} \frac{1}{2} \jump{ \,|\velhnp|^2\,},
\\
& =
\sum_{k=1}^{N_K} \int_{K_k} \div\brk{ \Ph\brk{\velhn} \frac{|\velhnp|^2}{2} }  - \sum_{k=1}^{N_K} \int_{\partial K_k} \brk{\Ph\brk{\velhn}\cdot\bn_{K_k}} \frac{|\velhnp|^2}{2}, 
\\
&= 0.
\end{split}
\]
Discrete free energy estimates for $(\velh,\ph,\lstrh)$ in $(\Po^{CR})^d\times\mathbb{P}_0\times(\Pz)^\dd$
can be similarly proven on the log-formulation.

\subsubsection{Alternative mixed finite element space for $(\velh,\ph)$ without inf-sup}
\label{sec:P0-noinfsup}

It is also possible to choose a mixed finite elements space for $(\velh,\ph)$
that does not satisfy the Babu\v{s}ka-Brezzi inf-sup condition,
like $(\mathbb{P}_{1})^d\times\mathbb{P}_0$ or $(\mathbb{P}_{1})^d\times\Po$.
The loss of stability due to the spaces' incompatibility
can then be alleviated through a stabilization procedure,
like Streamline Upwind Petrov Galerkin, Galerkin Least Square or Subgrid Scale Method (see~\cite{hughes-franca-87,codina-98,guermond-99}). In the following, we consider very simple stabilization procedures,
for which only one simple quadratic term is added to the variational finite element formulation
in order to restore stability of the discrete numerical scheme.

Let us first consider {\em the mixed finite element space $(\mathbb{P}_{1})^d\times\mathbb{P}_0$} for $(\velh,\ph)$.
If the term $\vel \cdot \grad \strs$ is discretized with the {\em characteristic} method, the system then writes:
\begin{equation} \label{eq:P0-charac-stab}
\begin{aligned}
0 =  \sum_{k=1}^{N_K} & \int_{K_k} \Re\brk{\frac{\velhnp-\velhn}{\dt} +
\velhn\cdot\gvelhnp}\cdot\bv 
 + \frac{\Re}{2} \div \velhn (\bv\cdot\velhnp)
\\
& \quad - \phnp\div\bv + q\div \velhnp+ (1-\e)\gvelhnp:\grad\bv
  +\frac{\e}{\Wi}\strhnp: \grad\bv
\\
& \quad + \brk{  \frac{\strhnp-{\strhn \circ X^n(t^n)} }{\dt} }:\bphi
- \brk{\brk{\gvelhnp}\strhnp+\strhnp\brk{\gvelhnp}^T}:\bphi
\\
& \quad + \frac{1}{\Wi}(\strhnp-\I):\bphi
\\
 + & \sum_{j=1}^{N_E} |E_j| \int_{E_j} \jump{\ph}\jump{q}\ ,
\end{aligned}
\end{equation}
with a flow $X^n$ computed with the projected field $\Phrot\brk{\velhn}$ 
through \eqref{eq:flow-P0-charac-HT}. The projection operator $\Phrot$ is the only one we can use among the three projectors we introduced in Section~\ref{sec:projections} because the weak incompressibility property \eqref{eq:weak-incompressibility} is not satisfied by $\velhn$.

The stabilization procedure  used in~\eqref{eq:P0-charac-stab} 
has been studied in~\cite{kechkar-silvester-92}.
Proposition \ref{prop:P0-charac} holds for system \eqref{eq:P0-charac-stab}, 
its proof being fully similar to the case of Taylor-Hood finite element (see Section~\ref{sec:P0-infsup}),
since the additional term $\sum_{j=1}^{N_E} |E_j| \int_{E_j}
\jump{\ph}\jump{q}$ is non negative with the test function used in the proof. All this also holds {\em mutatis mutandis} for discretization of the advection terms by a {\em discontinuous Galerkin} method, and for the log-formulation.

Let us finally consider {\em the mixed finite elements space $(\mathbb{P}_{1})^d\times\Po$} for $(\velh,\ph)$. If the term $\vel \cdot \grad \strs$ is discretized with the characteristic method, the system then writes:
\begin{equation} \label{eq:P0-charac-stab-bis}
\begin{aligned}
0 =  \intd & \Re\brk{\frac{\velhnp-\velhn}{\dt} +
\velhn\cdot\gvelhnp}\cdot\bv 
 + \frac{\Re}{2} \div \velhn (\bv\cdot\velhnp)
\\
& - \phnp\div\bv + q\div \velhnp+ (1-\e)\gvelhnp:\grad\bv
  +\frac{\e}{\Wi}\strhnp: \grad\bv
\\
& + \brk{  \frac{\strhnp-{\strhn \circ X^n(t^n)} }{\dt} }:\bphi
- \brk{\brk{\gvelhnp}\strhnp+\strhnp\brk{\gvelhnp}^T}:\bphi
\\
&+ \frac{1}{\Wi}(\strhnp-\I):\bphi
\\
 + &\sum_{k=1}^{N_K} h_{K_k}^2 \int_{K_k} \nabla\ph\cdot\nabla{q}\ ,
\end{aligned}
\end{equation}
with a flow $X^n$ again computed with the projected field $\Phrot\brk{\velhn}$ 
through \eqref{eq:flow-P0-charac-HT}.  Again, we are led to choose the projection operator $\Phrot$ because the weak incompressibility property \eqref{eq:weak-incompressibility} is not satisfied by $\velhn$. The stabilization procedure  used in~\eqref{eq:P0-charac-stab-bis} has been studied in~\cite{brezzi-pitkaranta-84}. Proposition~\ref{prop:P0-charac} holds for system~\eqref{eq:P0-charac-stab-bis}, 
its proof being fully similar to the case of Taylor-Hood finite element (see Section~\ref{sec:P0-infsup}),
since the additional term $\sum_{k=1}^{N_K} h_{K_k}^2 \int_{K_k}
\nabla\ph\cdot\nabla{q}$ is non negative with the test function used in the proof. All this also holds {\em mutadis mutandis} for discretization of the advection terms by a {\em discontinuous Galerkin} method, and for the log-formulation.

\begin{table}[!h]
\begin{center}
\begin{tabular}[t]{|p{4.5cm}|p{4.5cm}|}
\hline
$\bullet$ Advection discretized by: \newline
$\bullet (\velh,\ph)$ in $\dots$ ,
& Characteristics or DG. \newline $\Rightarrow$ equations modified:
\\
\hline \hline
Scott-Vogelius \newline
$(\Pt)^d\times\Pod$
& (nothing)
\\
\hline
Taylor-Hood \newline
$(\Pt)^d\times\mathbb{P}_{1}$
&
+ $\Phrot$\newline
+ Temam term
\\
\hline
Crouzeix-Raviart \newline
$(\Po^{CR})^d\times\mathbb{P}_{0}$
&
+ $\PhBDM$, $\PhRTO$ or $\Phrot$\newline 
+ $\Ph\brk{\velhn}$ for Navier term
\\
\hline
stabilized $(\Po)^d\times\mathbb{P}_{1}$
&
+ $\Phrot$ \newline
+ Temam term
\\
\hline
stabilized $(\Po)^d\times\mathbb{P}_{0}$
&
+ $\Phrot$\newline
+ Temam term
\\
\hline
\end{tabular}
\caption{\label{tab:finite element-summary} 
Summary of some possible finite elements for $(\velh,\ph,\strh/\lstrh)$
when $\strh/\lstrh \in (\Pz)^\dd$, 
with some possible projections for the velocity field (see Section~\ref{sec:need}).}
\end{center}
\end{table}

\section{Higher order discretization of the stress fields $\strh$ and $\lstrh$}
\label{sec:P1}

We now show how to build numerical schemes with higher order discretization spaces for the stress
that still satisfy a discrete free energy estimate. 
We typically have in mind piecewise linear spaces for $\strh$ and $\lstrh$.

{F}rom the previous proofs establishing discrete free energy estimates, 
it is clear that we need to use nonlinear functionals
of $\strh$ and $\lstrh$ as test functions, namely $\strh^{-1}$ and $\elstrh$. 
Finite element spaces other than $\Pz$ are typically not invariant under such nonlinear functionals, 
and this brings us to introduce projections of these nonlinear terms on $\Pz$.
A $\Pz$-Lagrange interpolation operator $\pi_h$ is convenient, because it commutes with nonlinear functionals 
(see Lemma~\ref{lemma:commute} below).

Choosing a $\Pz$-Lagrange interpolation operator for test functions will also have another important consequence,
because it will require that $\Pz$ be a subspace of the finite element space used to discretize the stress.
And our proofs establishing discrete free energy estimates will use in a crucial way the fact that
the interpolation operator coincides with a $L^2$ orthogonal projection onto $\Pz$ 
(see Lemma~\ref{lemma:lagrange} below).
The need for $\pi_h$ to coincide with a $L^2$ orthogonal projection onto $\Pz$
also {\it a priori} limits the maximum regularity of the discretization of the stress,
essentially to piecewise $\Po$ finite elements. Therefore, we consider $\strh$ and $\lstrh$ 
in either of the following finite element spaces\footnote{Note that, clearly, $(\Po + \Pz)^\dd$ is only a subspace of $(\Pod)^\dd$.}:
$$(\Po + \Pz)^\dd\text{   or   }(\Pod)^\dd\,.$$

In Section~\ref{sec:pih}, we introduce the interpolation operator $\pi_h$. 
Then we prove that, for a Scott Vogelius discretization of the velocity-pressure field, 
a free energy estimate can be obtained for discretization schemes close to those considered in Section~\ref{sec:P0},
when $\strh$ (respectively $\lstrh$) is in $(\mathbb{P}_{0})^\dd$. 
This is the purpose of the Section~\ref{sec:P1-strh} (respectively Section~\ref{sec:log-P1-strh}). 
Finally, we show in Section~\ref{sec:P1-need} how these results can be extended to 
other finite element discretizations of the velocity-pressure field.

\subsection{The interpolation operator $\pi_h$}\label{sec:pih}

Let us introduce the projection operator $\pi_h$ as the $\Pz$ Lagrange interpolation at barycenter
$\theta_{K_k}$ for each $K_k \in {\cal T}_h$. 
\begin{definition} \label{def:interpolation}
For $k=1\dots N_K$, we denote by $\theta_{K_k}$ the barycenter of the triangle $K_k$.
For any $\bphi$ such that $ \forall k=1\dots N_K$, $\bphi(\theta_{K_k})$ is well-defined (for example $\phi$ is a  tensor-valued function, continuous at points $\theta_{K_k}$), 
we define its piecewise constant interpolation by :
$$ \forall k=1\dots N_K\ , \pih{\bphi}|_{K_k} = \bphi(\theta_{K_k}).$$
\end{definition}
Notice that this definition also makes sense for the case in which $\bphi$ is
matrix-valued. And this interpolation operator $\pi_h$ coincides with the $L^2$ orthogonal 
projection from $(\Pod)^\dd$ onto $(\Pz)^\dd$,
as shown in the following Lemma proved in Appendix \ref{ap:lagrange}:
\begin{lemma}\label{lemma:lagrange}
Let $\pi_h$ be the interpolation operator introduced in Definition~\ref{def:interpolation}. Then, for any $\bphi_h \in (\Pod)^\dd$, we have :
\[
\intd \bphi_h : \tilde{\bphi_h} = \intd \pih{\bphi_h} : \tilde{\bphi_h},\ 
\forall \tilde{\bphi_h} \in (\Pz)^\dd.
\]
\end{lemma}

In addition, the following property holds, which is important in
the choice of this particular interpolation:
\begin{lemma}
\label{lemma:commute}
Let $\pi_h$ be the interpolation operator introduced in Definition~\ref{def:interpolation}.
The interpolation operator $\pi_h$ commutes with any function $f$: for any functions $f$ and $\bphi_h$ such that $\bphi_h$ and  $f(\bphi_h)$ are well-defined at the barycenters $\theta_k$,
\[
\pih{f(\bphi_h)} = f(\pih{\bphi_h}).
\]
\end{lemma}
The proof of Lemma \ref{lemma:commute} is straightforward
since, by Definition \ref{def:interpolation}, the interpolation $\pi_h$ only uses specific values at fixed points in the spatial domain $\D$.

\subsection{Free energy estimates with discontinuous piecewise linear $\strh$}\label{sec:P1-strh}

In this section, we consider the following finite element discretization:
Scott-Vogelius $(\Pt)^d\times\Pod$ for $(\velh,\ph)$ and $(\Pod)^\dd$ or $(\Po+\Pz)^\dd$ for $\strh$.

\subsubsection{The characteristic method}
\label{sec:P1-charac}

If the advection term $\vel \cdot \grad \strs$ is discretized by the {\em characteristic} method, the system writes:
\begin{equation} \label{eq:P1-charac}
\begin{split}
0 =  \intd & \Re\brk{\frac{\velhnp-\velhn}{\dt} +
\velhn\cdot\gvelhnp}\cdot\bv 
\\
& - \phnp\div\bv + q\div \velhnp+ (1-\e)\gvelhnp:\grad\bv
  +\frac{\e}{\Wi}\pih{\strhnp}:\grad\bv
\\
&+ \brk{  \frac{\strhnp-\pih{\strhn} \circ X^n(t^n) }{\dt} }:\bphi
- ( \gvelhnp \pih{\strhnp} +\pih{\strhnp} (\gvelhnp)^T ):\bphi
\\
&+ \frac{1}{\Wi}(\strhnp-\I):\bphi,
\end{split}
\end{equation}
where $X^n$ is defined as in \eqref{eq:flow-P0-charac}. Notice that we
have used the projection operator $\pi_h$ in four terms. It will become clearer from the proof of the free energy estimate below why those projections are needed.

\begin{proposition} \label{prop:P1-charac}
Let $(\velhn,\phn,\strhn)_{0 \le n \le N_T}$ be solution to~\eqref{eq:P1-charac}, such that $\pih{\strh^n}$ is positive definite.
Then, the free energy of the solution $(\velhn,\phn,\strhn)$:
\begin{equation} 
F_h^n=F(\velhn,\pih{\strhn}) = \frac{\Re}{2}\intd|\velhn|^2
+\frac{\e}{2 \Wi}\intd\tr\brk{\pih{\strhn}-\ln\pih{\strhn}-\I}\,,
\end{equation}
satisfies:
\begin{equation} \label{eq:free-energy-P1-charac}
\begin{split}
F_h^{n+1} - F_h^n & + \intd \frac{\Re}{2} |\velhnp-\velhn|^2 \\
& + \dt \intd (1-\e)|\gvelhnp|^2 + \frac{\e}{2 {\Wi}^2} \tr\brk{\pih{\strhnp}+\pih{\strhnp}^{-1}-2\I} \le 0.
\end{split}
\end{equation}
In particular, the sequence $(F_h^n)_{0 \le n \le N_T}$ is non-increasing.
\end{proposition}

\begin{remark}
The ensemble of symmetric positive definite matrices is convex. This implies that a piecewise linear tensor field is symmetric positive definite as soon as it is symmetric positive definite at the nodes of the mesh. Moreover, this also implies that $\pih{\strh}$ is symmetric positive definite as soon as $\strh$ is a piecewise linear (possibly discontinuous) symmetric positive definite tensor field.
\end{remark}
\begin{remark}
It is easy to extend the result of Proposition~\ref{prop:existence} to show that,
if $\dt>0$ is small enough and $\strhn$ is positive definite, 
then there exists a unique solution to~\eqref{eq:P1-charac} with $\strhnp$ positive definite. 
\end{remark}

\begin{proof}[Proof of Proposition \ref{prop:P1-charac}]
The test functions we choose are $\left(\velhnp,\phnp,\frac{\e}{2 \Wi}(\I - \pih{\strhnp}^{-1}\right)$. Recall that by Lemma~\ref{lemma:commute}, $\left(\pih{\strhnp}\right)^{-1}=\pih{(\strhnp)^{-1}}$. The proof is similar to the one of Proposition \ref{prop:P0-charac} except in the treatment of the 
constitutive equation.

The upper-convective term in the tensor derivative writes (using Lemma~\ref{lemma:lagrange} and the incompressibility property~\eqref{eq:divu=0}):
\begin{align*}
\intd \gvelhnp\pih{\strhnp}:(\I - \pih{\strhnp}^{-1}) 
= & \intd \pih{\strhnp} : \gvelhnp  \\ 
& \quad - \intd \gvelhnp\pih{\strhnp}:\pih{\strhnp}^{-1},\\
= & \intd \pih{\strhnp} : \gvelhnp -\intd \gvelhnp:\I, \\
= & \intd \pih{\strhnp} : \gvelhnp -\intd \div \velhnp, \\
= & \intd \pih{\strhnp} : \gvelhnp,
\end{align*}
which vanishes after combination with the extra-stress term in the momentum equation.

The last term rewrites (using again Lemma~\ref{lemma:lagrange}):
\[
\intd (\strhnp-\I):(\I-\pih{\strhnp}^{-1}) = \intd \tr(\pih{\strhnp}+\pih{\strhnp}^{-1}-2\I).
\]

The remaining term writes (using Lemma~\ref{lemma:lagrange}, the equation~\eqref{eq:traceAB} with $\strs=\pih{ \strhn }\circ X^n(t^n)$ and $\str=\pih{\strhnp}$, and the fact that the Jacobian of $X^n$ remains equal to one due to the incompressibility property~\eqref{eq:divu=0}):
\[
\begin{split}
\intd & \brk{\strhnp-\pih{ \strhn }\circ X^n(t^n) }:(\I-\pih{\strhnp}^{-1}) \\
& = \intd \tr \strhnp  - \tr \pih{\strhn} \circ X^n(t^n)  
 + \tr \Big( \pih{\strhn} \circ X^n(t^n) \pih{\strhnp}^{-1} -\I \Big), \\
& \ge \intd \tr \strhnp  - \tr \pih{\strhn} \circ X^n(t^n)
 + \tr \ln \pih{\strhn} \circ X^n(t^n) - \tr \ln \pih{\strhnp}, \\
& = \intd \tr\pih{\strhnp} - \tr\pih{\strhn} + \tr \ln \pih{\strhn} - \tr \ln \pih{\strhnp}. 
\end{split}
\]
This completes the proof.
\end{proof}

\subsubsection{The discontinuous Galerkin method}
\label{sec:P1-DG}

If the advection term $\vel \cdot \grad \strs$ is discretized by the {\em discontinuous Galerkin} method, the system writes:
\begin{equation} \label{eq:P1-DG}
\begin{split}
0 =  \sum_{k=1}^{N_K} \int_{K_k} & \Re\brk{\frac{\velhnp-\velhn}{\dt} + {\velhn}\cdot\gvelhnp}\cdot\bv 
\\
& - \phnp\div\bv + q\div \velhnp+ (1-\e)\gvelhnp:\grad\bv
  +\frac{\e}{\Wi}\pih{\strhnp}:\grad\bv
\\
&+ \brk{  \frac{\strhnp- \strhn}{\dt} }:\bphi
- \brk{ \gvelhnp\pih{\strhnp} +\pih{\strhnp}(\gvelhnp)^T }:\bphi
\\
&+ \frac{1}{\Wi}(\strhnp-\I):\bphi
\\
+ \sum_{j=1}^{N_E} & \int_{E_j} \Abs{\velhn\cdot\bn} \jump{\pih{\strhnp}} : \bphi^+.
\end{split}
\end{equation}
As for the characteristic method, the projection operator $\pi_h$ is
used in four terms.
Besides, like in the case where $\strh \in (\Pz)^\dd$, the advection term $\vel \cdot \grad \strs$ is discretized using a jump term only. Indeed, in order to derive discrete free energy estimates, we treat the discrete advection term using the projection $\pih{\strh} \in (\Pz)^\dd$ of the stress field $\strh$,
the derivative of which is zero. 

Proposition \ref{prop:P1-charac} still holds for the system \eqref{eq:P1-DG}.
The proof is straightforward using all the arguments of the previous sections,
except for the treatment of the discrete advection term for $\vel \cdot \grad \strs$.
Using the equations~\eqref{eq:traceAB}, \eqref{eq:jump}, the fact
that $\pih{\strh^{n+1}} \in (\Pz)^\dd$ and the weak incompressibility
property \eqref{eq:weak-incompressibility}, we have:
\begin{align*}
\sum_{j=1}^{N_E} \int_{E_j} & \Abs{\velhn\cdot\bn} \jump{\pih{\strhnp}}\left(\I-\pih{\strhnp}^{-1}\right)^+, \\
= &
\sum_{j=1}^{N_E} \int_{E_j} \Abs{\velhn\cdot\bn} \jump{\tr\pih{\strhnp}}
+
 \Abs{\velhn\cdot\bn} \tr \Big( \pih{\strh^{n+1,-}} \pih{\strh^{n+1,+}}^{-1} - \I\Big) ,\\
\ge &
\sum_{j=1}^{N_E} \int_{E_j} \Abs{\velhn\cdot\bn} \jump{\tr\pih{\strhnp}}
+
 \Abs{\velhn\cdot\bn} \tr \brk{ \ln \pih{\strh^{n+1,-}} - \ln \pih{\strh^{n+1,+}} }, \\
= &
\sum_{j=1}^{N_E} \int_{E_j} \Abs{\velhn\cdot\bn} \jump{\tr( \pih{\strh^{n+1}} - \ln \pih{\strh^{n+1}} )}, \\
= &
\sum_{k=1}^{N_K} - \int_{\partial K_k} \brk{{\velhn}\cdot\bn_{K_k}} \tr \Big( \pih{\strh^{n+1}} - \ln \pih{\strh^{n+1}} \Big), \\
= &
\sum_{k=1}^{N_K} - \tr\Big(\pih{\strh^{n+1}}-\ln\pih{\strh^{n+1}}\Big)\Big|_{K_k} \int_{K_k} \div({\velhn})\,, \\ 
= & 0.
\end{align*}

\subsection{Free energy estimates with discontinuous piecewise linear $\lstrh$}\label{sec:log-P1-strh}

In the following section, we write free-energy-dissipative schemes using the log-formulation
with $\lstrh$ piecewise linear. For this, we again need the projection
operator $\pi_h$ introduced in Definition~\ref{def:interpolation}.
We consider the Scott-Vogelius finite element space for $(\velh,\ph)$ 
and the following decomposition of the velocity gradient $\gvel_h\in(\Pod)^\dd$:
\begin{equation}
\gvel_h = \Omh + \B_h + \Nh \pih{\elstrh}^{-1}.
\end{equation}
Notice that since $\pih{\elstrh}^{-1} = e^{-\pih{\lstrh}}$ is in $(\Pz)^\dd$,
we have $\Omh,\B_h,\Nh \in (\Pod)^\dd$ 
in virtue of Lemma~\ref{lemma:discrete-decomposition} with $k=1$.

\subsubsection{The characteristic method}
\label{sec:log-P1-charac}

If the advection term $\vel \cdot \grad \strs$ is discretized by the {\em characteristic} method, the system writes:
\begin{equation} \label{eq:log-P1-charac}
\begin{split}
0 = \intd & \Re\brk{\frac{\velhnp-\velhn}{\dt} 
+ \velhn\cdot\gvelhnp}\cdot\bv 
\\
& - \phnp\div\bv + q\div \velhnp+ (1-\e)\gvelhnp:\grad\bv
  +\frac{\e}{\Wi}\pih{\elstrhnp}:\grad\bv
\\
&+ \brk{  \frac{\lstrhnp-\pih{\lstrhn} \circ X^n(t^n) }{\dt} }:\bphi
- ( \Omh^{n+1}\pih{\lstrhnp} - \pih{\lstrhnp}\Omh^{n+1} ):\bphi
\\
& -2\B_h^{n+1}:\bphi - \frac{1}{\Wi}(\pih{\emstrhnp}-\I):\bphi.
\end{split}
\end{equation}
In the system above, we have used the projection operator $\pi_h$ to
treat the same terms as in the system~\eqref{eq:P1-charac}. But in
addition to these, we have also used the projection operator for the
exponential term $\emstrhnp$ in the Oldroyd-B equation. 

\begin{proposition} \label{prop:log-P1-charac}
Let $(\velhn,\phn,\lstrhn)_{0 \le n \le N_T}$ be solution to~\eqref{eq:log-P1-charac}.
Then, the free energy of the solution $(\velhn,\phn,\lstrhn)$:
\begin{equation} 
F_h^n=F\brk{\velhn,e^{\pih{\lstrhn}}}  = \frac{\Re}{2}\intd|\velh|^2 
+ \frac{\e}{2 \Wi}\intd\tr\brk{e^{\pih{\lstrhn}} - \pih{\lstrhn} - \I}\,,
\end{equation}
satisfies:
\begin{equation} \label{eq:free-energy-log-P1-charac}
\begin{split}
F_h^{n+1} - F_h^n & + \intd \frac{\Re}{2} |\velhnp-\velhn|^2 \\
& + \dt \intd (1-\e)|\gvelhnp|^2 + \frac{\e}{2 {\Wi}^2} \tr\brk{e^{\pih{\lstrhn}}+e^{-\pih{\lstrhn}}-2\I} \le 0.
\end{split}
\end{equation}
In particular, the sequence $(F_h^n)_{0 \le n \le N_T}$ is non-increasing.
\end{proposition}

\begin{proof}[Proof of Proposition~\ref{prop:log-P1-charac}]
The proof is similar to that of Proposition~\ref{prop:log-P0-charac} except
for the terms using the interpolation operator $\pi_h$. We shall use as test functions 
$\brk{\velhnp,\phnp,\frac{\e}{2\Wi}(\pih{\elstrhnp}-\I)}$ in~\eqref{eq:log-P0-charac}. 
And we will make use of the following property all along the proof (see Lemma~\ref{lemma:commute}):
$$ \pih{\elstrhnp}  =  e^{\pih{\lstrhnp}}  .$$

For the material derivative of $\lstrh$, 
using Lemma~\ref{lemma:lagrange},
the equation~\eqref{eq:traceAB+} with $\strs=\elstrhnp$ and $\str=e^{\lstrhn \circ X^n(t^n)}$,
and the fact that the Jacobian of the flow $X^n$ is one for divergence free velocity field $\velhn$,
we have:
\begin{align*}
 \intd & \brk{\lstrhnp-\pih{\lstrhn}\circ X^n(t^n)}:(\pih{\elstrhnp}-\I) \\
& = \intd  \brk{\pih{\lstrhnp}-\pih{\lstrhn}\circ X^n(t^n)}:e^{\pih{\lstrhnp}}
 - \tr\brk{\pih{\lstrhnp}-\pih{\lstrhn}\circ X^n(t^n)}, \\
& \ge \intd \tr\brk{e^{\pih{\lstrhnp}}-\pih{\lstrhnp}}
 - \intd \tr\brk{e^{\pih{\lstrhn}}-\pih{\lstrhn}}\circ X^n(t^n),\\
& = \intd \tr\brk{e^{\pih{\lstrhnp}}-\pih{\lstrhnp}} 
 - \intd \tr\brk{e^{\pih{\lstrhn}}-\pih{\lstrhn}}.
\end{align*}

Besides, using the equation~\eqref{eq:omegapsi0}, we have:
\begin{align*}
\intd & ( \Omh^{n+1}\pih{\lstrhnp} - \pih{\lstrhnp}\Omh^{n+1} ):(e^{\pih{\lstrhnp}} -\I) \\
& = \intd ( \Omh^{n+1}\pih{\lstrhnp} - \pih{\lstrhnp}\Omh^{n+1} ):e^{\pih{\lstrhnp}},\\ 
& = 0,
\end{align*}
and using the equations~\eqref{eq:uB} and~\eqref{eq:divu=0}:
\begin{align*}
\intd \B_h^{n+1}:(\pih{\elstrhnp} -\I)
&=\intd \B_h^{n+1}:e^{\pih{\lstrhnp}} - \intd \div(\velhnp),\\
&=\intd \gvelhnp:e^{\pih{\lstrhnp}},
\end{align*}
which cancels out with the same term $\intd e^{\pih{\lstrhnp}}:\gvelhnp$ in the momentum equation.
\end{proof}

\subsubsection{The discontinuous Galerkin method}
\label{sec:log-P1-DG}

If the advection term $\vel \cdot \grad \strs$ is discretized by the {\em discontinuous Galerkin} method, the system writes:
\begin{equation} \label{eq:log-P1-DG}
\begin{split}
0 = \sum_{k=1}^{N_K} \int_{K_k} & \Re\brk{\frac{\velhnp-\velhn}{\dt} 
+ \velhn\cdot\gvelhnp}\cdot\bv 
\\
& - \phnp\div\bv + q\div \velhnp+ (1-\e)\gvelhnp:\grad\bv
  +\frac{\e}{\Wi}\pih{\elstrhnp}:\grad\bv
\\
&+ \brk{  \frac{\lstrhnp-\pih{\lstrhn}}{\dt} }:\bphi
- ( \Omh^{n+1}\pih{\lstrhnp} - \pih{\lstrhnp}\Omh^{n+1} ):\bphi
\\
& -2\B_h^{n+1}:\bphi - \frac{1}{\Wi}(\pih{\emstrhnp}-\I):\bphi
\\
 + \sum_{j=1}^{N_E} & \int_{E_j} \Abs{\velhn\cdot\bn} \jump{\pih{\lstrhnp}}:\bphi^+.
\end{split}
\end{equation}

Proposition~\ref{prop:log-P1-charac} still holds for solutions of the system~\eqref{eq:log-P1-DG}.
The proof follows that of the previous Section~\ref{sec:log-P1-charac} except for the treament of the jump term,
which follows that of the Section~\ref{sec:P0-strh-DG} (see also~\ref{sec:log-P0-DG}),
because $\pih{\lstrhnp} \in (\Pz)^\dd$ and
$\pih{\elstrhnp} = e^{\pih{\lstrhnp}}$ is also in $(\Pz)^\dd$.

\subsection{Other finite elements for $(\velh,\ph)$}\label{sec:P1-need}

In this section, we review the modifications that apply to the systems 
in the two previous Sections~\ref{sec:P1-strh} and~\ref{sec:log-P1-strh}
when the different mixed finite element spaces for $(\velh,\ph)$ proposed in Section~\ref{sec:need}
are used instead of Scott-Vogelius. 
Notice that the conclusions of Table~\ref{tab:P0-summary} 
about the conditions that the velocity field has to satisfy 
still hold for the two previous Sections~\ref{sec:P1-strh} and~\ref{sec:log-P1-strh}
with piecewise linear approximations of $\strh,\lstrh$.

Other finite elements space for $(\velh,\ph)$ than Scott-Vogelius
and adequate projections of the velocity field (see summary in Table~\ref{tab:finite element-summary})
have to be combined with interpolations of the stress field $\strh,\lstrh$ using $\pi_h$ 
(see the two previous Sections~\ref{sec:P1-strh} and~\ref{sec:log-P1-strh} above).
We give a summary of the $\pi_h$ projections that are then required in Table~\ref{tab:P1-summary}.

\subsubsection{Alternative mixed finite element space for $(\velh,\ph)$ with inf-sup condition}
\label{sec:P1-infsup}

The situation is very similar to that in Section~\ref{sec:P0-infsup}.
Among the mixed finite element space that satisfy the inf-sup condition, let us first choose the \emph{Taylor-Hood} $(\Pt)^d\times\Po$.
Again, because the velocity is not even weakly incompressible in the sense of the equation~\eqref{eq:weak-incompressibility}, 
we need to use the projection of the velocity field onto
the solenoidal vector fields for the treament of some terms in the variational formulations. 
When the advection terms $\vel\cdot\grad\strs$ and $\vel\cdot\grad\lstr$ are discretized 
using the \emph{characteristic} method,
we define the flow with $\Phrot\brk{\velhn}$ like in \eqref{eq:flow-P0-charac-HT} 
and use the same systems \eqref{eq:P1-charac} and~\eqref{eq:log-P1-charac} than above. 
When the advection terms $\vel\cdot\grad\strs$ and $\vel\cdot\grad\lstr$ are discretized 
using the \emph{discontinuous Galerkin} method,
we use systems similar to~\eqref{eq:P1-charac} and~\eqref{eq:log-P1-charac} above,
where the jump term rewrites (in the conformation-tensor formulation):
$$ + \sum_{j=1}^{N_E} \int_{E_j} |\Phrot\brk{\velhn}\cdot \bn| \jump{\pih{\strhnp}}:\bphi^+.$$
And one still needs to add the so-called Temam correction term~\eqref{eq:temam-term} to the weak formulation.

We can also use the \emph{Crouzeix-Raviart} finite elements for velocity (see~\eqref{eq:P1CR}): 
$(\velh,\ph,\strh)$ in $(\Po^{CR})^d\times\mathbb{P}_0\times(\Pod)^\dd$.
Similarly to the advection terms $\vel\cdot\grad\strs$ and $\vel\cdot\grad\lstr$,
the advection term $\vel\cdot\gvel$ in the Navier-Stokes equations should then be discretized 
either using a characteristic method with the flow defined in~\eqref{eq:P0-charac-CR} 
with any of the projections $\Ph$ introduced above for the velocity field,
or using the \emph{discontinuous Galerkin} method formulated in the equation~\eqref{eq:P0-DG-CR}.

It is noticeable that choosing the mixed finite elements of Crouzeix-Raviart
simplifies all the variational formulations presented above in the present Section~\ref{sec:P1}.
Indeed, since $\gvel\in(\Pz)^{d\times d}$ and we have the Lemma~\ref{lemma:lagrange},
it is then unnecessary to project the velocity except in the advection terms.
For instance, for the conformation-tensor formulation using the
discontinuous Galerkin method, the formulation writes:
\begin{equation} \label{eq:P1-DG-CR}
\begin{split}
0 = \sum_{k=1}^{N_K} \int_{K_k} & \Re\brk{\frac{\velhnp-\velhn}{\dt} + \Ph\brk{\velhn}\cdot\gvelhnp}\cdot\bv \\
& - \phnp\div\bv + q\div \velhnp+ (1-\e)\gvelhnp:\grad\bv  +\frac{\e}{\Wi}\strhnp:\grad\bv \\
& + \brk{  \frac{\strhnp-\pih{\strhn}}{\dt} }:\bphi - \brk{ (\gvelhnp)\strhnp +\strhnp(\gvelhnp)^T }:\bphi \\
& + \frac{1}{\Wi}(\strhnp-\I):\bphi \\
 + \sum_{j=1}^{N_E}  \int_{E_j} & \Abs{\Ph\brk{\velhn}\cdot\bn} \jump{\pih{\strhnp}} : \bphi^+
+ \Re \Abs{\Ph\brk{\velhn}\cdot\bn} \jump{\velhnp}\cdot\BRK{\bv}.
\end{split}
\end{equation}
Note that the second term in the sum of integrals over edges $E_j$ is due to the use of the Crouzeix-Raviart element, 
and is uncorrelated to the treatment of the advection by a discontinuous Galerkin method.

The discrete free energy estimate \eqref{eq:free-energy-P1-charac} holds. 
Its proof combines arguments of the proofs above, except for the treatment of the upper-convective term in~\eqref{eq:P1-DG-CR}.
This term writes, on any element $K_k$ of the mesh
(using Lemma~\ref{lemma:lagrange}, the fact that  $\gvel\in(\Pz)^{d\times d}$ and the incompressibility~\eqref{eq:divu=0}):
\[
\begin{split}
\int_{K_k} \gvelhnp\strhnp:(\I - \pih{\strhnp}^{-1}) 
& = \int_{K_k} \strhnp : \gvelhnp - \intd \strhnp:\pih{\strhnp}^{-1}\gvelhnp,\\
& = \int_{K_k} \pih{\strhnp} : \gvelhnp  - \intd \pih{\strhnp}:\pih{\strhnp}^{-1}\gvelhnp,\\
& = \int_{K_k} \pih{\strhnp} : \gvelhnp -\intd \div \velhnp, \\
& = \int_{K_k} \pih{\strhnp} : \gvelhnp,
\end{split}
\]
which vanishes after combination with the extra-stress term in the momentum equation, the latter satisfying:
\[
\int_{K_k} \strhnp:\gvelhnp = \int_{K_k} \pih{\strhnp}:\gvelhnp,
\]
because of the fact that  $\gvel\in(\Pz)^{d\times d}$ and using Lemma~\ref{lemma:lagrange}.

\subsubsection{Alternative mixed finite element space for $(\velh,\ph)$ without inf-sup}
\label{sec:P1-noinfsup}

It is also possible to use finite element spaces for $(\velh,\ph)$ that do not satisfy the 
inf-sup condition like in Section~\ref{sec:P0-noinfsup},
while the stress field is discretized using discontinuous piecewise linear approximations.
The construction of systems of equations and the derivation of discrete free energy estimates then
directly follow from the combination of results from the Section~\ref{sec:P0-noinfsup} with
those used above in Section~\ref{sec:P1}, after upgrading the degree of the polynomial approximations for the stress field.

If we consider the mixed finite element space $(\mathbb{P}_{1})^d\times\mathbb{P}_0$ for $(\velh,\ph)$,
and if the term $\vel \cdot \grad \strs$ is discretized with the {\em characteristic} method, the system then writes:
\begin{equation} \label{eq:P1-charac-stab}
\begin{aligned}
0 = \sum_{k=1}^{N_K} \int_{K_k} & \Re\brk{\frac{\velhnp-\velhn}{\dt} +
\velhn\cdot\gvelhnp}\cdot\bv 
 + \frac{\Re}{2} \div \velhn (\bv\cdot\velhnp)
\\
& - \phnp\div\bv + q\div \velhnp+ (1-\e)\gvelhnp:\grad\bv
  +\frac{\e}{\Wi}\strhnp: \grad\bv
\\
& + \brk{  \frac{\strhnp-{\pih{\strhn} \circ X^n(t^n)} }{\dt} }:\bphi
- \brk{\brk{\gvelhnp}\strhnp+\strhnp\brk{\gvelhnp}^T}:\bphi
\\
& + \frac{1}{\Wi}(\strhnp-\I):\bphi
\\
 + \sum_{j=1}^{N_E} & |E_j| \int_{E_j} \jump{\ph}\jump{q}\ ,
\end{aligned}
\end{equation}
with a flow $X^n$ computed with the projected field $\Phrot\brk{\velhn}$ through \eqref{eq:flow-P0-charac-HT}.
It is noteworthy that, for the same reason than above in the equation~\eqref{eq:P1-DG-CR},
the projection operator $\pi_h$ is needed only for the discretization of the advection term $\vel\cdot\grad\strs$.

If we consider the mixed finite element space $(\mathbb{P}_{1})^d\times\mathbb{P}_1$ for $(\velh,\ph)$,
it is straightforward to rewrite the system~\eqref{eq:P0-charac-stab-bis} where the stress field was only piecewise constant,
while using the same argument as above to see that only the advection term for the stress field needs a projected velocity.

\begin{remark}
We were not able to retrieve discrete free energy estimates without interpolating some terms 
in the formulations above thanks to the operator $\pi_h$. 
This operator projects the stress $\strh$ (or $\lstrh$) onto~$(\Pz)^\dd$. 
Thus, for the formulations we have considered in this section, the interest of using larger
dimensional spaces for $\strh$ (or $\lstrh$) than $(\Pz)^\dd$ is not clear. 
Our aim in this section is simply to exhibit discrete formulations with piecewise
linear approximations of the stress, for which we are able to derive a free energy estimate.
\end{remark}

\begin{table}[!h]
\begin{center}
\begin{tabular}[t]{|p{2cm}|p{4.5cm}|p{6.5cm}|}
\hline
$\gvel \dots$ & $\strh \in (\Pod)^\dd$ & $\lstrh \in (\Pod)^\dd$ \\
\hline
\hline
in $(\Pz)^{d^2}$ 
& $\pih{\strhn}$ in time derivative 
\newline (incl. flux term in DG)
& $\pih{\lstrhn}$ in time derivative
\newline (incl. flux term in DG)
\newline + implicit source term $\pih{\emstrhnp}$ in OB
\newline + implicit coupling term $\pih{\elstrhnp}$ in NS 
\\
\hline
not in $(\Pz)^{d^2}$
& $\pih{\strhn}$ in time derivative 
\newline (incl. flux term in DG)
\newline 
\newline + implicit coupling terms
\newline ( $\pih{\strhnp}$ in NS, OB )
& $\pih{\lstrhn}$ in time derivative
\newline (incl. flux term in DG)
\newline + implicit source term $\pih{\emstrhnp}$ in OB
\newline + implicit coupling terms
\newline ( $\pih{\elstrhnp}$ in NS, $\pih{\lstrhnp}$ in OB )
\\
\hline
\end{tabular}
\caption{\label{tab:P1-summary} 
Summary of projected terms in the Navier-Stokes (NS) and Oldroyd-B (OB) equations
for $\strh/\lstrh$ in $(\Pod)^\dd$.}
\end{center}
\end{table}


\section{Positivity, free energy estimate and the long-time issue}
\label{sec:stability}

Notice that both Propositions~\ref{prop:existence} and~\ref{prop:log-existence} 
impose a limitation on the time step 
which depends on the time iteration: $0 < \dt < c_0$,
where $c_0 \equiv c_0(\velhn,\strhn)$ is function of a time-dependent data.
Thus, these existence results are weak
insofar as the long-time existence of the discrete solutions is not
insured, \textit{i.e.} if $\sum_{n=0}^\infty c_0(\velhn,\strhn) < \infty$.

Yet, for the discretizations introduced above, we have also shown that at each time step,
the solutions of those discretizations satisfy free energy estimates. This will now
allow us to prove the long-time existence of the discrete solutions:
\begin{proposition} \label{prop:stability}
For any initial condition $(\velh^0,\strh^0)$ with $\strh^0$ symmetric positive definite, 
there exists a constant $c_1 \equiv c_1\brk{\velh^0,\strh^0} > 0$ such that, 
for any time step $0 \le \dt < c_1$, 
there exists, for all iterations $n\in\mathbb{N}$, $(\velh^{n+1},\strh^{n+1})$
which is the unique solution to the system~\eqref{eq:P0-charac} (resp.~\eqref{eq:P0-DG})
with $\strh^{n+1}$ symmetric positive definite.
\end{proposition}
\begin{proof}[Proof of Proposition~\ref{prop:stability}]
Like in the proof of Proposition~\ref{prop:existence}, we will proceed with 
the proof for system~\eqref{eq:P0-charac} only, using its rewriting as system~\eqref{eq:P0-charac-proj}.

The proof is by induction on the time index $n$. With the notation of the proof of Proposition~\ref{prop:existence},
for a fixed $n=0\dots N_T-1$ and for a fixed  vector $Y^n$ of values in the subset $S_+^*$ of $\R^{2N_D+3N_K}$ 
(standing for the nodal and elementwise values of a field $(\velhn,\strhn)$ with
$\strhn$ symmetric positive definite), 
we define like in the proof of Proposition~\ref{prop:existence} (using the implicit function theorem)
a function $R:\dt\in[0,c_0)\to R(\dt)\in\R^{2N_D+3N_K}$ (where $c_0=c_0(\velhn,\strhn)$) such that:
$$ \forall \dt\in[0,c_0), Q(\dt,R(\dt)) = 0,$$
where $Q$ is defined by~\eqref{eq:Q}. For any $\dt\in[0,c_0)$, $R(\dt)\in\R^{2N_D+3N_K}$ stands for the nodal and elementwise values of a field $(\velh(\dt),\strh(\dt))$ (with $\strh(\dt)$ symmetric positive definite)
that is solution to the system~\eqref{eq:P0-charac-proj}.

Then, by Proposition~\ref{prop:P0-charac}, the solution $(\velh(\dt),\strh(\dt))$ satisfies a free energy estimate:
\begin{equation} \label{eq:argument_estimate}
F(\velh(\dt),\strh(\dt)) \leq F(\velhn,\strhn) \,. 
\end{equation}

Using the fact that all norms are equivalent in the finite-dimensional
vector space $\R^{2N_D+3N_K}$, and that, for $0<\nu\le 1-\frac1e$, we
have $ \nu \ x \le x - \ln(x) \,,\forall x>0 $, we obtain that there
exists two constants $\alpha>0$ and $\beta>0$ (independent of $\dt$), such that:
\begin{equation} \label{eq:argument_norm}
\alpha \|R(\dt)\| \leq F(\velh(\dt),\strh(\dt)) + \beta \,. 
\end{equation}

Let us define the function $D$:
$$ D: \dt \in [0,c_0) \longrightarrow B(Y^n) + A(R(\dt)) +(\nabla_Y A)R(\dt)  \in \R^{2N_D+3N_K}.$$
We recall that (see~\eqref{eq:grad_Q}), with the new notations:
$\nabla_Y Q(\dt,R(\dt))=I + \dt D(\dt)$.
Using~\eqref{eq:argument_estimate},~\eqref{eq:argument_norm} and the
fact that the discrete free energy is non-increasing, the function $D$ satisfies: 
\begin{align*}
\|D(\dt)\| &\le \|B\| \|Y^n\| + (\|A\|+\|\nabla_Y A\|) \|R(\dt)\|,\\
& \le (\|B\|+\|A\|+\|\nabla_Y A\|) \frac{1}{\alpha}
\brk{F(\velhn,\strhn)+\beta},\\
& \le (\|B\|+\|A\|+\|\nabla_Y A\|) \frac{1}{\alpha}
\brk{F(\velh^0,\strh^0)+\beta}.
\end{align*} 
This shows that there exists a constant $c_1 \equiv c_1\brk{\velh^0,\strh^0} > 0$ such that,  for any time step $0 \le \dt < c_1$, the matrix $\nabla_Y Q(\dt,R(\dt))$ is invertible. Using the implicit function theorem, this implies that, for any time step $0 \le \dt < c_1$, the system~\eqref{eq:P0-charac-proj} admits a solution $(\velhnp,\strhnp)$  with $\strhnp$ symmetric positive definite at all iterations $n\in\mathbb{N}$.
\end{proof}

A similar result can be proven for the log-formulations~\eqref{eq:log-P0-charac} and~\eqref{eq:log-P0-DG}:
\begin{proposition} \label{prop:log-stability}
For any initial condition $(\velh^0,\lstrh^0)$, 
there exists a constant $c_1 \equiv c_1\brk{\velh^0,\lstrh^0} > 0$ such that, 
for any time step $0 \le \dt < c_1$, 
there exists, for all iterations $n\in\mathbb{N}$, $(\velh^{n+1},\lstrh^{n+1})$
which is the unique solution to the system~\eqref{eq:log-P0-charac} (resp.~\eqref{eq:log-P0-DG}).
\end{proposition}
\begin{proof}[Proof of Proposition~\ref{prop:log-stability}]
The proof of Proposition~\ref{prop:log-stability} is fully similar to that of Proposition~\ref{prop:stability}
using for $Q(\dt,Y)$ and $D(\dt)$ slightly modified expressions 
as explained for the proof of Proposition~\ref{prop:log-existence}.
The entropic term in the free energy still helps at bounding the norm of the vector of nodal-elementwise values
for $(\velh,\lstrh)$ like in~\eqref{eq:argument_norm} using the following scalar
inequality, which is true for any fixed $\nu \in (0,1]$: $\forall x \in \R$, $ e^x-x+1 \ge \nu |x|$.
\end{proof}

{F}rom Propositions~\ref{prop:stability} and~\ref{prop:log-stability}, 
we have the global-in-time existence of solutions to those
discretizations of the Oldroyd-B system  presented above which satisfy a discrete free energy estimate.

The log-formulation actually also satisfies the following long-time
existence result, using the fact that the {\it a priori} estimates can be obtained
without requiring the stress tensor field to be positive definite:
\begin{proposition} \label{prop:log-stability-2}
For any initial condition $(\velh^0,\lstrh^0)$, 
and for any constant time step $\dt > 0$, 
there exists, for all iterations $n\in\mathbb{N}$, $(\velh^{n+1},\lstrh^{n+1})$
which is solution to the system~\eqref{eq:log-P0-charac} (resp.~\eqref{eq:log-P0-DG}).
\end{proposition}
\begin{proof}[Proof of Proposition~\ref{prop:log-stability-2}]
We will proceed with the proof for system~\eqref{eq:log-P0-charac} only, 
using its rewriting as system~\eqref{eq:log-P0-charac-proj}.
Note already that, since the derivation of discrete free energy estimates for the system~\eqref{eq:log-P0-charac}
does not require the solution $\lstrhnp$ and the test function to be non-negative like in 
the derivation of discrete free energy estimates for the system~\eqref{eq:P0-charac}, 
then the manipulations used to derive the free energy estimate~\eqref{eq:free-energy-log-P0-charac} 
can also be done {\it a priori} for any function in the finite element space. 

Let us consider a fixed time index $n$ and a given couple
$(\velhn,\lstrhn)\in (\Pt)^d_{\div=0}\times(\Pz)^\dd$. We equip the Hilbert space $(\Pt)^d_{\div=0}\times(\Pz)^\dd$ with the following inner product:
$$ \brk{ (\bv_1,\bphi_1) ; (\bv_2,\bphi_2) } = \intd \bv_1\cdot\bv_2 + \bphi_1:\bphi_2 \,, $$
for all $(\bv_1,\bphi_1),(\bv_2,\bphi_2) \in
(\Pt)^d_{\div=0}\times(\Pz)^\dd$, and denote by $\|\cdot\|$ the
associated norm. Let us introduce
the mapping ${\cal F}:
(\Pt)^d_{\div=0}\times(\Pz)^\dd \to 
(\Pt)^d_{\div=0}\times(\Pz)^\dd$ defined by: for $(\vel, \lstr) \in (\Pt)^d_{\div=0}\times(\Pz)^\dd$,  for any test function $(\bv,\bphi) \in (\Pt)^d_{\div=0}\times(\Pz)^\dd$,
\begin{align*}
\brk{ \mathcal{F}(\vel, \lstr) ; (\bv,\bphi) } 
= \intd & \Re\brk{\frac{\vel-\velhn}{\dt} + \velhn\cdot\gvel}\cdot\bv 
+ (1-\e)\gvel:\grad\bv +\frac{\e}{\Wi}\elstr:\grad\bv
\\
& + \brk{  \frac{\lstr-\lstrhn \circ X^n(t^n) }{\dt} }:\bphi
- ( \Om\lstr - \lstr\Om ):\bphi
-2\B:\bphi
\\
& - \frac{1}{\Wi}(\mlstr-\I):\bphi,
\end{align*}
where $\Om$ and $B$ are associated with the decomposition of the
velocity gradient $\gvel$ as explained in Lemma~\ref{lemma:decomposition}:
\begin{equation*}
\gvel = \Om + \B + \N \mlstr.
\end{equation*}
If $(\velhnp,\lstrhnp)$ is solution to~\eqref{eq:log-P0-charac-proj},
then we have: for all $(\bv,\bphi) \in (\Pt)^d_{\div=0}\times(\Pz)^\dd$,
\begin{equation}\label{eq:solution}
\brk{\mathcal{F}(\velhnp,\lstrhnp); (\bv,\bphi) }  = 0.
\end{equation}

Considering $\cal F(\bv,\bphi)$ 
and using $\brk{\bv,\frac{\e}{2\Wi}(e^{\bphi} -\I)}$ as test functions, 
we get the following inequality after similar manipulations to those in the 
proof of Proposition~\ref{prop:log-P0-charac}:
\begin{equation}\label{eq:inegalite}
\begin{split}
& \brk{ \mathcal{F}(\bv,\bphi) ; \brk{\bv,\frac{\e}{2\Wi}(e^{\bphi} -\I)} } \\
& \quad \ge \frac{\Re}{2}\intd|\bv|^2 + \frac{\e}{2 \Wi}\intd\tr(e^{\bphi} - \bphi) 
 - \frac{\Re}{2}\intd|\velhn|^2 - \frac{\e}{2 \Wi}\intd\tr(\elstrhn - \lstrhn ) \\
& \quad\quad + \intd\frac{\Re}{2} |\bv-\velhn|^2 
 + \dt \intd (1-\e)|\gvelhnp|^2 + \frac{\e}{2 {\Wi}^2} \tr\brk{e^{\bphi}+e^{-\bphi}-2I}.
\end{split}
\end{equation}

Let us now assume that, for any $\alpha>0$, the mapping $\cal F$ has no zero $(\velhnp,\lstrhnp)$ 
satisfying~\eqref{eq:solution} in the ball
$$ \mathcal{B}_\alpha = \BRK{ (\bv,\bphi) \in (\Pt)^d_{\div=0}\times(\Pz)^\dd ,\,  \|(\bv,\bphi)\|\le\alpha } .$$
Then, for such a real number $\alpha > 0$ to be fixed later, we define the following continuous mapping from $\mathcal{B}_\alpha$ onto itself:
$$ \mathcal{G}(\bv,\bphi) = -\alpha \frac{\mathcal{F}(\bv,\bphi)}{\|\mathcal{F}(\bv,\bphi)\|},
\forall (\bv,\bphi) \in (\Pt)^d_{\div=0}\times(\Pz)^\dd. $$
By the Brouwer fixed point theorem, $\cal G$ has a fixed point in $\mathcal{B}_\alpha$.
Let us still denote that fixed point $(\bv,\bphi)$ for the sake of simplicity. By definition, it satisfies:
\begin{equation}\label{eq:fixpoint}
\mathcal{G}(\bv,\bphi) = (\bv,\bphi) \in \mathcal{B}_\alpha \text{ and } \|\mathcal{G}(\bv,\bphi)\| = \alpha.
\end{equation}

Using the equation~\eqref{eq:fixpoint} with the inequality~\eqref{eq:inegalite}, we have:
\begin{align}
\nonumber
& \brk{ \mathcal{F}(\bv,\bphi) ; \brk{\bv,\frac{\e}{2\Wi}(e^{\bphi} -\I)} } \\
\nonumber
& \quad = -\frac{ \| \mathcal{F}(\bv,\bphi) \| }{\alpha} \brk{ (\bv,\bphi) ; \brk{\bv,\frac{\e}{2\Wi}(e^{\bphi} -\I)} }, \\
\nonumber
& \quad = -\frac{ \| \mathcal{F}(\bv,\bphi) \| }{\alpha} 
  \brk{ \intd |\bv|^2 + \frac{\e}{2 \Wi} \tr\brk{\bphi\, e^{\bphi} - \bphi} }, \\
\nonumber
& \quad \ge \frac{\Re}{2}\intd|\bv|^2 + \frac{\e}{2 \Wi}\intd\tr(e^{\bphi} - \bphi ) 
 - \frac{\Re}{2}\intd|\velhn|^2 - \frac{\e}{2 \Wi}\intd\tr(\elstrhn - \lstrhn ) \\
\label{eq:statement1}
& \quad\quad + \intd\frac{\Re}{2} |\bv-\velhn|^2 
 + \dt \intd (1-\e)|\gvelhnp|^2 + \frac{\e}{2 {\Wi}^2} \tr\brk{e^{\bphi}+e^{-\bphi}-2I}.
\end{align}

Using the scalar inequality $e^x-x\ge|x|,\, \forall x\in\R$, we have:
\begin{equation}\label{eq:scalar1}
\intd \tr(e^{\bphi} - \bphi + \I) \ge \sum_{i=1}^d \intd |\lambda_i|, \, \forall \bphi\in(\Pz)^\dd, 
\end{equation}
where $(\lambda_i)_{1\le i\le d}$ are functions depending on $\bphi$ such that, for all $\x \in\mathcal{D}$, 
$(\lambda_i(\x))_{1\le i\le d}$ are the $d$ (non-necessarily distinct) real eigenvalues of 
the symmetric matrix $\bphi(\x)$.

Now, since $ (\Pt)^d_{\div=0}\times(\Pz)^\dd $ is finite-dimensional, all norms are equivalent.
So there exist $\gamma_1,\gamma_2 > 0$ such that, for all $(\bv,\bphi) \in (\Pt)^d_{\div=0}\times(\Pz)^\dd$:
\begin{equation}\label{eq:norm_equivalence}
\gamma_1\|(\bv,\bphi)\| 
\le \brk{\intd|\bv|^2}^\frac12 +\|\max_{1\le i\le d} |\lambda_i(\x)|\|_{\infty}
\le \gamma_2\|(\bv,\bphi)\|\,,
\end{equation}
where it is easy to prove that $\|\max_{1\le i\le d}|\lambda_i(\x)|\|_{\infty}$ defines a norm in the vector space
$L^\infty\brk{\D,\mathcal{S}(\R^{d\times d})}$.

Using the equation~\eqref{eq:scalar1} with the norm equivalence~\eqref{eq:norm_equivalence}, we obtain:
\begin{align}
\frac{\Re}{2} & \intd|\bv|^2 + \frac{\e}{2 \Wi}\intd\tr(e^{\bphi} - \bphi + \I)
\\
\ge
& \min\brk{\frac{\Re}{2},\frac{\e}{2 \Wi}\frac{1}{\|\max_{1\le i\le d}|\lambda_i(\x)|\|_{\infty}}}
\brk{ \intd|\bv|^2 + \|\max_{1\le i\le d}|\lambda_i(\x)|\|_{\infty} \sum_{i=1}^d \intd |\lambda_i| }\,,
\\
\ge 
& \min\brk{\frac{\Re}{2},\frac{\e}{2 \Wi}\frac{1}{\|\max_{1\le i\le d}|\lambda_i(\x)|\|_{\infty}}} 
\brk{ \intd|\bv|^2 + \sum_{i=1}^d \intd |\lambda_i|^2 }\,,
\\
= 
& \min\brk{\frac{\Re}{2},\frac{\e}{2 \Wi}\frac{1}{\|\max_{1\le i\le d}|\lambda_i(\x)|\|_{\infty}}} 
\|(\bv,\bphi)\|^2\,,
\\
\ge
& \min\brk{\frac{\Re}{2},\frac{\e}{2 \Wi \gamma_2 \alpha}} \alpha^2.
\end{align}
If we now choose $\alpha$ large enough so that:
$$ \min\brk{\frac{\Re}{2},\frac{\e}{2 \Wi \gamma_2 \alpha}} \alpha^2 > 
\frac{\Re}{2}\intd|\velhn|^2 + \frac{\e}{2 \Wi}\intd\tr(\elstrhn - \lstrhn + \I), $$
we get:
\begin{equation}\label{eq:statement2}
\begin{split}
& \frac{\Re}{2}\intd|\bv|^2 + \frac{\e}{2 \Wi}\intd\tr(e^{\bphi} - \bphi + \I) 
 - \frac{\Re}{2}\intd|\velhn|^2 - \frac{\e}{2 \Wi}\intd\tr(\elstrhn - \lstrhn + \I) \\
& \quad\quad + \intd\frac{\Re}{2} |\bv-\velhn|^2 
 + \dt \intd (1-\e)|\gvelhnp|^2 + \frac{\e}{2 {\Wi}^2} \tr\brk{e^{\bphi}+e^{-\bphi}-2\I} > 0.
\end{split}
\end{equation}
This is obviously in contradiction with~\eqref{eq:statement1} since, for all $ \bphi\in(\Pz)^\dd$, we have
$ \tr(\bphi e^{\bphi} - \bphi) \ge 0$ in virtue of the scalar inequality $x(e^x-1)\ge0,\,\forall x\in\R$.

Thus, for any $\dt > 0$, if we choose $\alpha$ sufficiently large, the mapping $\cal F$ has a zero $(\velhnp,\lstrhnp)$ satisfying~\eqref{eq:solution} in the ball $\mathcal{B}_\alpha$.
\end{proof}
Notice that Proposition~\ref{prop:log-stability-2} does not ensure
the uniqueness of solutions. The fact that we are able to prove
such a stability result without any assumption on the timestep for the
log-formulation, and not for the classical formulation, may be related
to the fact that the log-formulation has been reported to be more stable
than the classical formulation (see~\cite{hulsen-fattal-kupferman-05}).

\begin{remark}[Other positivity preserving schemes]
\label{sec:lie}
There exist other means than using the log-formulation to preserve the non-negativity of the conformation tensor.
A very natural way of preserving the non-negativity is to use a Lie-formulation
of the upper convective derivative term.
For instance, following~\cite{lee-xu-06}, one can write the following discretization of~\eqref{eq:form-sigma}
using Scott-Vogelius elements for $(\velh,\ph)$ and piecewise constant approximations for $\strh$:
\begin{equation} \label{eq:P0-lie}
\begin{split}
0 = \intd & \Re\brk{\frac{\velhnp-\velhn}{\dt} + \velhn\cdot\gvelhnp}\cdot\bv 
\\
& - \phnp\div\bv + q\div \velhnp+ (1-\e)\gvelhnp:\grad\bv
  +\frac{\e}{\Wi}\strhnp:\grad\bv
\\
&+\brk{\frac{\strhnp-(\I-\dt\pih{\gvelhnp})^{-1}\brk{ \strhn \circ X^n(t^n) }(\I-\dt\pih{\gvelhnp})^{-T} }{\dt}}:\bphi
\\
&+ \frac{1}{\Wi}(\strhnp-\I):\bphi,
\end{split}
\end{equation}
where the function $X^n(t)$ is defined by~\eqref{eq:flow-P0-charac}.
The system~\eqref{eq:P0-lie} admits a solution such that
$(\I-\dt\pih{\gvelhnp})^{-1}$ is well-defined, provided $\dt$ is sufficiently small
(but possible very small when $\Norm{\gvelhnp}$ is large).
Besides, taking $\bphi$ as the characteristic function of some element $K_k$, 
we have the following equality inside $K_k$:
\begin{equation}\label{eq:local_identity}
\brk{1+\frac{\dt}{\Wi}}\strhnp =  (\I-\dt\pih{\gvelhnp})^{-1}\brk{ \strhn \circ X^n(t^n) }(\I-\dt\pih{\gvelhnp})^{-T}
+\frac{\dt}{\Wi}\I.
\end{equation}
Then it is clear that the system~\eqref{eq:P0-lie} preserves the non-negativity of $\strhn$.
Moreover, it is possible to derive the free energy
estimate~\eqref{eq:free-energy-P0-charac} for the
system~\eqref{eq:P0-lie}. It suffices to take as a test function for the stress:
$$ \bphi= \frac{\e}{2\Wi} \brk{\I-\dt\pih{\gvelhnp}}^T \brk{\I-(\strhnp)^{-1}} (\I-\dt\pih{\gvelhnp}) \,, $$
and to proceed to the derivation of a free energy estimate using both ideas of the present work
and of the work~\cite{lee-xu-06}, after noting that:
\begin{equation} \label{eq:noting}
\tr\brk{ \pih{\gvelhnp}^T \brk{\I-(\strhnp)^{-1}} \pih{\gvelhnp} 
\brk{ \brk{1+\frac{\dt}{\Wi}}\strhnp - \frac{\dt}{\Wi}\I } } \ge 0 \,,
\end{equation}
the proof of which is fully similar to the proof of~\eqref{eq:AB},
using the fact that 
$\brk{1+\frac{\dt}{\Wi}}\strhnp - \frac{\dt}{\Wi}\I$
is symmetric positive definite (provided $\dt$ is sufficiently small)
and $\pih{\gvelhnp}^T \brk{\I-(\strhnp)^{-1}} \pih{\gvelhnp}$
is symmetric positive semi-definite.
\end{remark}


\appendix


\section{Some properties of symmetric positive definite matrices}

\subsection{Proof of Lemma \ref{lemma:alg_spd} }\label{ap:alg_spd}

Formula~\eqref{eq:trln},~\eqref{eq:tr-ln} and~\eqref{eq:tr-plus_inv} are
simply obtained by diagonalizing the symmetric positive definite matrix $\strs$, and using the
inequalities: $\forall x,y >0$, $\ln(xy)=\ln x + \ln y$, $x-1 \ge
\ln x$ and $x + 1/x \ge 2$.

Let us now prove Formula~\eqref{eq:AB}.
By diagonlization, we have
$\strs=\Omega^T D \Omega$ with $\Omega$ orthogonal and $D$ diagonal strictly positive,
which gives:
\begin{align}
\tr(\strs \str^{-1}) & = \tr(\Omega^T \sqrt{D}\sqrt{D} \Omega\str^{-1}),
\\
& = \tr(\sqrt{D} \Omega \str^{-1}\Omega^T \sqrt{D}),
\\
& \ge 0,
\end{align}
because $\sqrt{D} \Omega \str^{-1}\Omega^T \sqrt{D}$ is clearly a symmetric positive definite matrix.

Likewise, we have:
\begin{eqnarray*}
\det(\strs \str^{-1}) & = & \det(\Omega^T D \Omega \str^{-1}), \\
& = & \det(\sqrt{D} \Omega \str^{-1}\Omega^T \sqrt{D}).
\end{eqnarray*}
And noting that $A=\sqrt{D} \Omega \str^{-1}\Omega^T \sqrt{D}$ is symmetric positive definite,
the proof of~\eqref{eq:traceAB} is then equivalent to show:
$$ \ln(\det(A)) \leq \tr(A-\I) \,, $$
for any symmetric positive definite $A$, which derives from~\eqref{eq:trln} and~\eqref{eq:tr-ln}.

It remains to prove Formula~\eqref{eq:traceAB+}.
By diagonalization, we write $\strs=O^T D O$ and $\str=R^T \Lambda R$
with $\Omega$ and $R$ orthogonal, and $D$ and $\Lambda$ diagonal
strictly positive. Let us introduce the orthogonal matrix
$\Omega=OR^T$. We denote by $D_i$ (resp. $\Lambda_i$) is the $(i,i)$-th
entry of $D$ (resp. of $\Lambda$). 

Let us first prove Formula~\eqref{eq:traceAB+} in dimension 2. We set
$\phi=\Omega_{11}^2$ ($\phi \in [0,1]$), so that
$\Omega_{12}^2=(1-\phi)$, $\Omega_{21}^2=(1-\phi)$ and
$\Omega_{22}^2=\phi$. We have
\begin{align*}
J(\phi)&=\tr\brk{\brk{ \ln \strs - \ln \str} \strs - \strs + \str},\\
&=\sum_{i} D_i \ln D_i - D_i + \Lambda_i - \sum_{i,j} \Omega_{ij}^2 D_i
\ln \Lambda_j,\\
&=\sum_{i} D_i \ln D_i - D_i + \Lambda_i - \phi \sum_i D_i
\ln \Lambda_i - (1 - \phi) \sum_{i \neq j} D_i \ln \Lambda_j.
\end{align*}
The functional $J$ is an affine function defined on the interval
$[0,1]$. For $\phi=0$, we have
\begin{align*}
J(0)&=\sum_{i} D_i \ln D_i - D_i + \Lambda_i  - \sum_{i \neq j} D_i \ln
\Lambda_j,\\
&=\sum_{i} D_i \brk{ \frac{\Lambda_{s(i)}}{D_i} - 1 - \ln \brk{ \frac{\Lambda_{s(i)}}{D_i} } } \ge 0,
\end{align*}
for $s$ the permutation of indices $1 \leftrightarrow 2$. We have used
the scalar inequality $x-1 \ge \ln x$ for all $x > 0$. Likewise, for $\phi=1$,
\begin{align*}
J(1)&=\sum_{i} D_i \ln D_i - D_i + \Lambda_i  - \sum_i D_i \ln
\Lambda_i,\\
&=\sum_{i} D_i \brk{ \frac{\Lambda_{i}}{D_i} - 1 - \ln \brk{ \frac{\Lambda_{i}}{D_i} } } \ge 0.
\end{align*}
The functional $J$ is thus non negative on $[0,1]$, which ends the proof
in dimension 2.

More generally, in dimension $d$, let us define the following affine functional:
\begin{align*}
J( \{ \Omega_{ij}^2 \} ) &= \tr\brk{\brk{ \ln \strs - \ln \str} \strs - \brk{\strs - \str} }, \\
&= \sum_{i} D_i \ln D_i - D_i + \Lambda_i - \sum_{i,j} \Omega_{ij}^2 D_i \ln \Lambda_j.
\end{align*}
The functional $J$ is defined on the convex polyhedron
$\mathcal{P} = \left\{\{\Omega_{ij}^2\} \in \R_+^{d^2}, 
\forall i \sum_j \Omega_{ij}^2 = 1,
\forall j \sum_i \Omega_{ij}^2 = 1 \right\}$,
and 
$\mathcal{Q} = \BRK{\{\Omega_{ij}^2\} \in \R_+^{d^2}, \Omega \Omega^T = \I}$
is a subset of $\cal P$.
All maxima and minima of $J$ on $\cal P$ are reached at vertices of the
convex polyhedron $\mathcal{P}$ (Theorem 10.3-2 in \cite{ciarlet-85}).
Besides, $\{\Omega_{ij}^2\} $ is a vertex (that is, an extremal point) of $\cal P$
if, and only if, for all $1\le i\le d$, there exists one, and only one, 
$j=s(i)$, $1\le j\le d$, such that $\Omega_{is(i)}^2\ne 0$
(Theorem 10.3-1 in \cite{ciarlet-85}).
Thus, the vertices of $\cal P$ are exactly the $ \{\Omega_{ij}^2\}$ such
$$ \Omega_{ij}^2 = \left\{ 
\begin{array}{l}
1,\ j=s(i) \\
0,\ j\ne s(i)
\end{array}
\right. $$ 
where $s$ is  a permutation on $\{1, \ldots, d\}$. Now, for any permutation $s$, we have:
\[ 
\begin{split}
J & = \sum_{i} D_i \ln D_i - D_i + \Lambda_i - D_i \ln \Lambda_{s(i)} \\
& = \sum_{i} D_i \brk{ \frac{\Lambda_{s(i)}}{D_i} - 1 - \ln \brk{ \frac{\Lambda_{s(i)}}{D_i} } } \ge 0,
\end{split}
\]
using the scalar inequality $x-1 \le \ln x$ for all $x > 0$. 
Hence the result: $J( \{ \Omega_{ij}^2 \} )\ge0$, $\forall \{ \Omega_{ij}^2 \} \in \mathcal{Q}$.

\subsection{Proof of Lemma \ref{lemma:dtrdt}}\label{ap:dtrdt} 


Since $\strs \in C^1([0,T))$ is symmetric positive definite, it can be decomposed as: $\forall t \in [0,T)$
\[
\strs(t) = R(t)^T \Lambda(t) R(t)
\]
where the unitary matrix $R$ and the diagonal matrix $\Lambda$ are
continuously differentiable.
Let us compute the following derivatives:
\begin{eqnarray*}
\deriv{}{t} \ln \strs = \deriv{R^T}{t} \ln\Lambda R + R^T\deriv{\Lambda}{t}\Lambda^{-1}R +
R^T \ln\Lambda \deriv{R}{t},
\\
\deriv{}{t} \strs = \deriv{R^T}{t} \Lambda R + R^T\deriv{\Lambda}{t}R +
R^T \Lambda \deriv{R}{t}.
\end{eqnarray*} 
Thus, we obtain:
\[
\tr\brk{\deriv{}{t} \ln \strs} = \tr\left(\deriv{\Lambda}{t}\Lambda^{-1}\right) +
\tr\left(\ln\Lambda \deriv{}{t} (R R^T)\right)
\]
and
\[
 \tr \left( \deriv{\strs}{t} \strs^{-1} \right) = \tr\left(\deriv{\Lambda}{t}\Lambda^{-1}\right) +
\tr \left(\deriv{}{t}(R^TR)\right).
\]
Now, for all times $t$,
\[
R^TR \equiv \I,
\]
hence~\eqref{eq:tracelog}. The proof of \eqref{eq:trace} is similar.

\section{Proof of Lemma~\ref{lem:spd}}\label{ap:spd}

Let us introduce
$$t_0=\inf \{ t >0, \, \strs \text{ is not symmetric positive definite} \},$$
with convention $t_0=\infty$ if $\{ t >0, \, \strs \text{ is not symmetric positive definite} \}=\emptyset$.
Since $\strs(t=0)$ is symmetric positive definite, it remains so at least for small times $0
\leq t < \dt$, by continuity of the eigenvalues with respect to the
time variable $t$. Thus, $t_0 \ge \dt > 0$.

Let us assume that $t_0 < \infty$. First, one can define the logarithm $\ln \strs$ of the matrix $\strs$,
which satisfies the equation for $\lstr$ in system~\eqref{eq:oldroyd-b-log} for $t \in [0,t_0)$.
Taking the trace of the equation for $\lstr$ in system~\eqref{eq:oldroyd-b-log}, we get for $\ln \strs$:
\begin{equation} \label{lem:spd-eq1}
 \frac{D}{Dt}\ln \det \strs = \frac{1}{\Wi} \tr(\strs^{-1} -\I),
\end{equation}
where we have introduced the convective derivative
$$\frac{D}{Dt}=\left(\deriv{}{t} + (\vel \cdot \grad) \right).$$
Besides, for any positive definite matrix $\strs^{-1}$,
we have:
\begin{equation} \label{lem:spd-eq2}
\frac{ \tr(\strs^{-1}) }{ d } \geq (\det \strs^{-1})^{1/d},
\end{equation}
which follows from the convex inequality between geometrical and
arithmetical means. Thus, combining \eqref{lem:spd-eq1} and
\eqref{lem:spd-eq2}, we get, on the time interval $[0,t_0)$:
\begin{equation} \label{eq:detsigma}
  \frac{D}{Dt} (\det \strs)^{1/d} = \frac1d (\det \strs)^{1/d}   \frac{D}{Dt} \ln \det \strs \\
 \geq \frac{1}{\Wi} \brk{ 1 - (\det \strs)^{1/d} }.
\end{equation}

Now, by continuity of the eigenvalues with respect to $t$, 
one eigenvalue at least converges to zero as $t \rightarrow t_0^-$, 
which implies $\det \strs \rightarrow 0^+$.
Then, there exists $\eta > 0$ such that, for times $t_0 - \eta < t < t_0$, we have:
$$ 0 < \det \strs < 1, $$
and by \eqref{eq:detsigma}
\begin{equation}\label{eq:bibi} 
\frac{D}{Dt}(\det \strs)^{1/d} > 0.
\end{equation}
But then, $t_0$ cannot be the first time when $\det \strs = 0$, 
otherwise one should have $ \frac{D}{Dt}(\det \strs)^{1/d}(t_0^-) \leq
0$, which contradicts~\eqref{eq:bibi}. Thus $t_0=\infty$ which ends the
proof of Lemma~\ref{lem:spd}.

\section{Proof of Lemmas~\ref{lemma:decomposition} and~\ref{lemma:discrete-decomposition}}\label{app:decomposition}

Let $\strs$ be any piecewise constant function with symmetric positive definite matrix values 
$\strs \in \mathcal{S}_+^*((\Pz)^{d\times d})$. 
We define the following linear applications for functions $N$ with skew-symmetric real matrix values 
$N \in \mathcal{A}(\R^{d\times d})$:
\begin{align}
f_\strs & : N \rightarrow \frac12 (N\strs^{-1}-\strs^{-1}N) \\
g_\strs & : N \rightarrow \frac12 (N\strs^{-1}+\strs^{-1}N).
\end{align}
The linear application $g_\strs$ is an automorphism in $\mathcal{A}(\R^{d\times d})$ and 
the linear application $f_\strs$
is a one-to-one onto function from $\mathcal{A}(\R^{d\times d})$ to the set of symmetric real matrices with null trace $\strs \in \mathcal{S}^0((\Pz)^{d\times d})$.Both applications are of rank $1$. 

Now, for any $\x \in \D$, the symmetric positive definite matrix $\strs(\x)$ can be diagonalized 
with an orthogonal matrix $O(\x) \in \mathcal{O}(\R^{d\times d})$ ($O^TO=\I=OO^T$)
and a diagonal positive definite matrix $D(\x) \in \mathrm{D}((\R_+^*)^d)$:
$$ \strs(\x) = O(\x)^T D(\x) O(\x). $$
So, if we change the basis, we can equivalently work with a diagonal positive definite matrix $\strs(\x)\in \mathrm{D}((\R_+^*)^d)$.

Let $\velhnp$ be some vector field in $(\mathbb{P}_k)^2$ with $k \in \mathbb{N}$ and
$\lstrhnp \in \mathcal{S}\brk{(\Pz)^\dd}$ an elementwise constant symmetric rank-two-tensor field.
The rank-two-tensor field $\gvelhnp$ takes its values in $(\mathbb{P}_{k-1,disc})^{d\times d}$.
For $\x \in \D$, $\gvelhnp(\x)$ decomposes elementwise 
in the basis where the matrix of the linear application 
associated with the symmetric matrix $\lstrhnp \in \mathcal{S}(\Pz)^\dd$ is diagonal.
More precisely, in this basis, $\gvelhnp(\x)$ uniquely decomposes as a direct sum in the supplementary spaces (of respective dimension $1$, $2$ and $1$):
$$ \R^{d\times d} 
= \mathcal{A}(\R^{d\times d})\oplus\mathrm{D}(\R^d)\oplus\mathcal{S}^0(\R^{d\times d}). $$
This decomposition reads:
\begin{align}
\gvelhnp(\x) = & \brk{ \frac{\gvelhnp(\x) + \gvelhnp(\x)^T}{2} - \mathrm{diag}\brk{\gvelhnp(\x)} }
 \\
& + \mathrm{diag}\brk{\gvelhnp(\x)} + \frac{\gvelhnp(\x) - \gvelhnp(\x)^T}{2} .
\end{align}

Since $\emstrhnp$ is an elementwise constant diagonal positive definite matrix,
in order to get the following decomposition at each $\x \in \D$:
\begin{align}
\gvelhnp(\x) & = \Omh^{n+1}(\x) + \B_h^{n+1}(\x) + \Nh^{n+1}(\x)\emstrhnp(\x) \\
& = \frac{\Nh^{n+1}(\x)e^{-\lstrhnp(\x)} - e^{-\lstrhnp(\x)}\Nh^{n+1}(\x)}{2} + \B_h^{n+1}(\x) \\
& \quad + \Omh^{n+1}(\x) + \frac{\Nh^{n+1}(\x)e^{-\lstrhnp(\x)}+e^{-\lstrhnp(\x)}\Nh^{n+1}(\x)}{2},
\end{align}
we need to do the following identification:
\begin{gather}
\B_h^{n+1}(\x) = \mathrm{diag}\brk{\gvelhnp(\x)}, \\
\Nh^{n+1}(\x) = f^{-1}_{e^{-\lstrhnp(\x)}}\brk{ \frac{\gvelhnp(\x) + \gvelhnp(\x)^T}{2} -\mathrm{diag}\brk{\gvelhnp(\x)} },
 \\
\Omh^{n+1}(\x) =  \frac{\gvelhnp(\x) - \gvelhnp(\x)^T}{2} - g_{e^{-\lstrhnp(\x)}}(\Nh^{n+1}(\x)),
\end{gather}
which is unique because of the requirement that $\B_h^{n+1}(\x)$ commutes with the diagonal
matrix $e^{-\lstrhnp(\x)}$.  

With the choice above, the matrices $\Omh^{n+1},\B_h^{n+1},\Nh^{n+1}$ have entries in $(\mathbb{P}_{k-1,disc})^{2\times 2}$ because $e^{-\lstrhnp(\x)}\in (\Pz)^\dd$.
Moreover, we clearly have: $\tr(\B_h^{n+1}) = \div\velhnp$. Besides, notice that the triplet ($\Omh^{n+1},\B_h^{n+1},\Nh^{n+1}$)
is unique provided the change of basis for $\emstrhnp$ to be diagonal is unique,
which means provided that $\emstrhnp$ is not singular (all the eigenvalues should be distinct one another).

\section{Proof of Lemma~\ref{lemma:lagrange}}\label{ap:lagrange}

For any $\bphi_h \in (\Pod)^\dd$, we want to show that:
\[
\intd \bphi_h : \tilde{\bphi_h} = \intd \pih{\bphi_h} : \tilde{\bphi_h},\ 
\forall \tilde{\bphi_h} \in (\Pz)^\dd.
\]
It is enough to prove Lemma~\ref{lemma:lagrange} on each simplex $K_k \in {\cal T}_h$
and in the scalar case.
Let $(x_i)_{1\leq i\leq 3}$ be the vertices of the
simplex $K_k$ and $(\psi_i)_{1\leq i\leq 3}$ the corresponding (linear) basis
functions in $\Po$. 
Then, the function $\bphi_h|_{K_k} \in \Po$ reads
\[
\bphi_h|_{K_k}(x) = \bphi_h(x_1)\psi_1(x) + \bphi_h(x_2)\psi_2(x) +
\bphi_h(x_3)\psi_3(x),\ \forall x \in K_k.
\]
For every $\tilde{\bphi_h} \in \Pz$,
\[
\int_{K_k} \bphi_h \tilde{\bphi_h} 
= \tilde{\bphi_h} \left( \int_{K_k} \bphi_h \right) 
= \tilde{\bphi_h}\frac{|K_k|}{3} \left( \bphi_h(x_1) + \bphi_h(x_2) + \bphi_h(x_3) \right)
\]
because $\int_{K_k} \psi_i = \frac{|K_k|}{3}$. 
Moreover, $\bphi_h|_{K_k} \in \Po$, hence
\[
\frac{1}{3} \left( \bphi_h(x_1) + \bphi_h(x_2) + \bphi_h(x_3) \right)
=\bphi_h\left( \frac{x_1+x_2+x_3}{3} \right)=\bphi_h(\theta_{K_k})
\]
which means
\[
\int_{K_k} \bphi_h \tilde{\bphi_h} = \int_{K_k} \tilde{\bphi_h} \bphi_h(\theta_{K_k})
= \int_{K_k} \pih{\bphi}\tilde{\bphi_h}.
\]


\paragraph*{Acknowledgements} We would like to thank A.~Ern and J.W.~Barrett for fruitful discussions.

\bibliography{biblio_HD,biblio_MHD,ma_biblio,biblio_MicMac}
\bibliographystyle{plain}

\end{document}